\documentclass{amsart}	

\usepackage{fullpage}

\usepackage{amssymb, amsmath, amsthm, color, amsfonts, caption}

\newtheorem{theorem}{Theorem}
\numberwithin{theorem}{section}
\newtheorem{corollary}[theorem]{Corollary}
\newtheorem{lemma}[theorem]{Lemma}

\newtheorem{proposition}[theorem]{Proposition}

\newtheorem*{remark}{Remark}

\newcommand{\leg}[2]{\left(\frac{#1}{#2}\right)}

\newcommand{\G}{\Gamma}

\newcommand{\C}{\mathbb{C}}
\newcommand{\R}{\mathbb{R}}
\newcommand{\Z}{\mathbb{Z}}
\renewcommand{\H}{\mathbb{H}}

\newcommand{\sm}[4]{\left(\begin{smallmatrix} #1 & #2 \\ #3 & #4\end{smallmatrix}\right)}

\newcommand{\ds}{\displaystyle}
\allowdisplaybreaks

\begin{document}

\author{Amanda Folsom, Sharon Garthwaite, Soon-Yi Kang, Holly Swisher, Stephanie Treneer}

\title{Quantum mock modular forms arising from eta-theta functions}

\maketitle

\begin{abstract} In 2013, Lemke Oliver classified all eta-quotients which are theta functions.  In this paper, we unify the eta-theta functions by constructing
mock modular forms from the eta-theta functions with even characters, such that the shadows of these mock modular forms are given by the eta-theta functions with odd characters.   In addition, we prove that our mock modular forms are quantum modular forms.  As corollaries, we establish simple finite hypergeometric expressions which may be used to evaluate Eichler integrals of the odd eta-theta functions, as well as some curious algebraic identities.
\end{abstract}

\section{Introduction and Statement of Results} \label{intro}
One of the most well-known modular forms of weight $1/2$ is Dedekind's $\eta$-function, defined for $\tau$ in the upper half-plane $\mathbb H:=\{\tau \in \mathbb C \ | \ \text{Im}(\tau)>0\}$ by
\begin{align}
\label{eqn_etadef}\eta(\tau) := q^{\frac{1}{24}}\prod_{n=1}^\infty (1-q^n) = \sum_{m \geq 1}     \Big(\frac{12}{m}\Big)q^{\frac{m^2}{24}},
\end{align}
where $q=e(\tau)$, and $\left(\frac{\cdot}{\cdot}\right)$ is the Kronecker symbol (throughout we set $e(u):=e^{2\pi i u}$).   
More generally, \emph{eta-products}, functions of the form
\begin{align}\label{def_etaprod} \prod_{j=1}^c\eta(a_j \tau)^{b_j},\end{align}
where $a_j, b_j,$ and $c$ are positive integers, have been of interest not only within the classical theory of modular forms, but also in connection to the representation theory of finite groups. Conway and Norton \cite{CN} showed that many character generating functions for the ``Monster" group $\mathbb M$, the largest of the finite sporadic simple groups, could be realized as \emph{eta-quotients}, which are of the same form as the functions in (\ref{def_etaprod}), but allow negative integer exponents $b_j$.  Mason \cite{Mason} similarly exhibited many character generating functions for the Mathieu group $\mathbb M_{24}$ as {\emph{multiplicative}} eta-products, meaning their $q$-series have multiplicative coefficients, as seen in (\ref{eqn_etadef})  for example.  This relationship to character generating functions in part motivated Dummit, Kisilevsky and McKay \cite{DKM} to classify all multiplicative eta-products.  Later, Martin \cite{Martin} classified all multiplicative integer weight eta-quotients.

In addition to being a simple example of a multiplicative $q$-series, the right-most function in (\ref{eqn_etadef}) is also an example of a \emph{theta function}, which is of the form \begin{align}\label{eqn_thetapsi}\theta_\chi(\tau) := \sum_n \chi(n) n^\nu q^n,\end{align}  where $\chi$ is an even (resp. odd) Dirichlet character, and $\nu$ equals $0$ (resp. 1).  The sum in (\ref{eqn_thetapsi}) is taken over $n \in \mathbb Z$ or $n \in \mathbb N$, depending on whether or not $\chi$ is trivial. It is well-known that such functions are ordinary modular forms of weight $1/2 + \nu$.  
While (\ref{eqn_etadef}) shows an eta-quotient which is also a theta function, it is not true in general that all multiplicative eta-quotients are also theta functions.  This question was studied by Lemke Oliver \cite{LO}, who classified all eta-quotients which are also theta functions;  in particular, his classification gives six odd \emph{eta-theta functions} $E_m$, and eighteen even eta-theta functions $e_n$ (some of which are twists by certain principal characters).  By ``odd (resp. even) eta-theta function", we mean an eta-quotient which is also a theta function with odd (resp. even) character. See Section \ref{LOprelim} for more on these functions.

Modular theta functions also naturally emerge in the theory of harmonic Maass forms, which  are certain non-holomorphic functions that transform like modular forms.  
 Harmonic Maass forms $\widehat{M}$, as originally defined by Bruinier and Funke \cite{BF}, naturally decompose into two parts as $\widehat{M} = M  + M^-$, where $M$ is the \emph{holomorphic part} of $\widehat{M}$, and $M^-$ is the \emph{non-holomorphic part} of $\widehat{M}$.  
For example, when viewed as a function of $\tau$, we now know that the function
\begin{align}\label{def_fqmock}
q^{-\frac{1}{24}} \sum_{n\geq 0}  \frac{q^{n^2}}{(-q;q)_n^2} + 2 i \sqrt{3} \int_{-\overline{\tau}}^{i\infty} \frac{\sum_{n \in \mathbb Z}   \left(n+\tfrac16\right) e^{3\pi i z\big (n+\tfrac16\big )^2  }}{\sqrt{-i(z+\tau)}} dz,
\end{align}
where $(a;q)_n:=\prod_{j=0}^{n-1} (1-a q^j)$, is a harmonic Maass form due to work of Zwegers \cite{Zwegers2}.   Its holomorphic part, namely the $q$-hypergeometric series in (\ref{def_fqmock}), is one of Ramanujan's original \emph{mock theta functions},  certain curious $q$-series whose exact modular properties were unknown for almost a century.  Beautifully, all of Ramanujan's original mock theta functions turn out to be examples of holomorphic parts of harmonic Maass forms \cite{Zwegers2}, and we now define after Zagier \cite{ZagierB} a \emph{mock modular form} to be any holomorphic part of a harmonic Maass form.  Mock modular forms come naturally equipped with a \emph{shadow}, a certain modular cusp form related via a differential operator,  which we formally define in Section \ref{sub:construction}.   In the example given above in (\ref{def_fqmock}), it turns out that the shadow of Ramanujan's $q$-hypergeometric mock theta function is essentially the modular theta function  given in the numerator of the integral appearing there (up to a simple multiplicative factor).
 
As mentioned above, character generating functions for $\mathbb M_{24}$ appear as multiplicative eta-products. We now also know that there is a rich Moonshine phenomenon surrounding mock modular forms.  Eguchi, Ooguri and Tachikawa \cite{EOT}, in analogy to the original Moonshine conjectures, observed that certain characters for the Mathieu group $\mathbb M_{24}$ appeared to be related to mock modular forms.  Their work was later generalized and greatly extended by Cheng, Duncan, and Harvey \cite{CDH}, who developed an ``umbral Moonshine" theory.  Their umbral Moonshine conjectures were recently proved by Duncan, Griffin, and Ono \cite{DGO}.

These connections serve as motivation for the first set of results in this paper.  In Section \ref{constructingV}, we unify the eta-theta functions  by constructing  mock modular forms which encode them  in the following ways.  We define functions $V_{mn}$ using the even eta-theta functions $e_n$, and in Theorem \ref{thm_modvij}, we prove that these functions $V_{mn}$ are mock modular forms, with the additional property that their shadows are given by the odd eta-theta functions $E_m$.

 To describe these results, we introduce some notation.   The functions $V_{mn}$ are indexed by pairs $(m,n)$ where $m \in T':=\{1,2,3,4,4',4'',5,6\}$ and $n\in \mathbb N$, and the admissible values for $n$ are dependent on $m$.  Throughout, we will call a pair $(m,n)$ \emph{admissible} if it is used to index one of our functions $V_{mn}$. In total, there are 59 admissible pairs $(m,n)$ where $m\in T'$.  When we restrict $m\in T:=\{1,2,3,4,5,6\}$, a particular subset which we also consider, there are a total of 43 admissible pairs $(m,n)$.   We provide a complete list of these functions in the Appendix.  The groups $A_{mn}$, integers $k^{(mn)}_{\gamma_{mn}}, \ell^{(mn)}_{\gamma_{mn}}, r^{(mn)}_{\gamma_{mn}},$ and $s^{(mn)}_{\gamma_{mn}}$, and roots of unity $\varepsilon_{\gamma_{mn}}^{(m)}$ appearing in Theorem \ref{thm_modvij} below and throughout are defined in Section \ref{constructingV}; the root of unity $\psi$ is defined in Lemma \ref{lem_etatransf}, and the constants $c_m$ are defined in Section \ref{sec_quantum}.
 
\begin{theorem}\label{thm_modvij}
 For any admissible pair $(m,n)$ with $m \in T'$, the functions $V_{mn}$  are mock modular forms of weight $1/2$ with respect to the congruence subgroups $A_{mn}$.  Moreover, for $m\in T$, the shadow of $V_{mn}$ is given by a constant multiple of the odd eta-theta function $E_m\left(\frac{2\tau}{c_m^2}\right)$.   In particular, the functions $V_{mn}, m\in T'$, may be completed to  form harmonic Maass forms $\widehat{V}_{mn}$ of weight $1/2$ on $A_{mn}$, which satisfy for all $\gamma_{mn} = \sm{a_{mn}}{b_{mn}}{c_{mn}}{d_{mn}} \in A_{mn}$, and $\tau \in \mathbb H$,
$$ \widehat{V}_{mn}(\gamma_{mn}\tau) =  \psi(\gamma_{mn})^{-3} (-1)^{k^{(mn)}_{\gamma_{mn}} + \ell^{(mn)}_{\gamma_{mn}} +  r^{(mn)}_{\gamma_{mn}} + s^{(mn)}_{\gamma_{mn}}}\varepsilon_{\gamma_{mn}}^{(m)}(c_{mn}\tau+d_{mn})^{\frac12}  \widehat{V}_{mn}(\tau).$$ \end{theorem}

Because there are infinitely many mock modular forms with a given shadow, we are additionally motivated to construct our functions $V_{mn}$ so that they are in some sense canonical.  One way of doing this is by utilizing the even eta-theta functions $e_n$  in the construction of these functions, as we have already mentioned.  Further, we show in Theorem \ref{thm_quantum} that our mock modular forms $V_{mn}$ are also \emph{quantum} modular forms, a property that is not necessarily true of all mock modular forms.  A quantum modular form, as defined by Zagier \cite{Zagier} in 2010, is a complex function defined on an appropriate subset of the rational numbers, as opposed to the upper half-plane, which transforms like a modular form, up to the addition of an error function that is suitably continuous or analytic in $\mathbb R$.  (See Section \ref{sec_quantum} for more detail.)  The theory of quantum modular forms is in its beginning stages;  constructing explicit examples of these functions remains of interest, as does answering the question of how quantum modular forms may arise from mock modular forms (see the recent articles \cite{BFR2, BOPR, FOR}, for example).

The quantum sets $S_{mn}$ of rational numbers and groups $G_{mn}$ pertaining to the forms $V_{mn}$ are defined in Section \ref{sec_quantumsets} and Section \ref{sec_quantum}, and the constants $\ell_m, a_m, b_m, c_m$ and $\kappa_{mn}$ appearing below are defined in Section \ref{sec_quantum}.  Here and throughout, the numbers $\ell_{mn}$ are defined to equal $\ell_m$ or $2$, depending on whether or not $n=1$.  For $N\in\mathbb N$, we define $\zeta_N := e(1/N)$, and for $r\in \mathbb Z$, we let $M_r:=\sm{1}{0}{r}{1}$.
\begin{theorem}\label{thm_quantum}  For any admissible pair $(m,n)$ with $m \in T$,   the functions $V_{mn}$  are quantum modular forms of weight $1/2$ on the sets $S_{mn}\backslash\left\{\frac{-1}{\ell_{mn}}\right\}$ for the groups $G_{mn}$.
In particular, the following are true. \medskip \ \\
\textnormal{(i)} For all $x\in\H\cup S_{mn}\backslash\left\{\frac{-1}{2}\right\}$,  we have that
$$V_{mn}(x)+\zeta_{4}^{\ell_m}(2 x+1)^{-\frac{1}{2}}V_{mn}\left(M_{2} x\right) =-\frac{i}{c_m}\int_{\frac{1}{2}}^{i\infty} \frac{E_m \big(\tfrac{2u}{c_m^2}\big )}{\sqrt{-i(u+x)}}\;du.$$
\noindent \textnormal{(ii)}  For $n=1$ and $m\in\{2,4,6\}$, for all $x\in\H\cup S_{mn}\backslash\left\{-1\right\}$, we also have that
\begin{align}\label{eqn_FEminvertthm1}V_{m1}(x)-\zeta_{8}^{ -1 } 
(  x+1)^{-\frac{1}{2}}V_{m1}\left(M_{1} x\right)=-\frac{i}{c_m}\int_{1}^{i\infty} \frac{E_m \big(\tfrac{2u}{c_m^2}\big )}{\sqrt{-i(u+x)}}\;du. \end{align}  
\noindent \textnormal{(iii)} For all $x\in \H\cup S_{mn}$, we have that
\begin{align}\label{eqn_FEminvertthm2}  
 V_{mn}(x) - \zeta_{a_m}^{\kappa_{mn}}  V_{mn}(x+ \kappa_{mn} b_m)   &= 0.
 \end{align}
\end{theorem}
One interesting feature of Theorem \ref{thm_quantum} is that it leads to simple, yet non-obvious, closed expressions for the Eichler integrals  of the odd eta-theta functions $E_m$ appearing on the right hand side of  (\ref{eqn_FEminvertthm1}).  Moreover, (\ref{eqn_FEminvertthm2}) leads to curious algebraic identities.   To describe these results, we define the truncated $q$-hypergeometric series for integers $h \in \mathbb Z, k \in \mathbb N$ $(\gcd(h,k)=1)$ by
\begin{align}\label{def_Fhk}F_{h,k}(z_1,z_2) := \sum_{n=0}^{k-1}  \frac{(-\zeta_{2k}^h;\zeta_{2k}^h)_n \zeta_{4k}^{n(n+1)h}}{(z_1; \zeta_{2k}^h)_{n+1}(z_2; \zeta_{2k}^h)_{n+1}}.\end{align}
The additional constants $d_m, H_m = H_m(h,k),$ and $K_m=K_m(h,k)$ appearing in Corollary \ref{cor_hypergeometric} below are defined in Section \ref{sec_cor}.
From Theorem \ref{thm_quantum}, we have the following corollary.
\begin{corollary}\label{cor_hypergeometric}
The Eichler integrals of the odd eta-theta functions $E_m$ satisfy the following identities. \medskip \\
\noindent \rm{(i)} Let $m\in\{1,2,3,5,6\}$, and $\frac{h}{k} \in S_{m1} \setminus \{\frac{-1}{\ell_m}\}$. Then we have that
\begin{align}\label{eqn_cor1}
   \frac{-i}{c_m}  \int_{\frac{1}{\ell_m}}^{i\infty}  \frac{E_m(\frac{2z}{c_m^2})}{\sqrt{-i\left(z+\tfrac{h}{k}\right)}}dz =
  &  i^{1+\ell_m}  \zeta_{a_m c_m k}^{2 d_m h} F_{h,k}(-i^{\ell_m-3} \zeta_{c_m k}^h,-i^{3-\ell_m} \zeta_{a_m k}^{d_m h})
\\ \notag &  \hspace{.15in}  - \zeta_{8}^{-5\ell_m} \zeta_{a_m c_m K_{\ell_m}}^{2 d_m H_{\ell_m}} \left( \tfrac{\ell_m h}{k}+1\right)^{-\frac12}
F_{H_{\ell_m},K_{\ell_m}}(-i^{\ell_m - 3 } \zeta_{c_m K_{\ell_m}}^{H_{\ell_m}},-i^{3-\ell_m} \zeta_{a_m K_{\ell_m}}^{d_m H_{\ell_m}}).
  \end{align}
\noindent Moreover, we have for $m\in\{1,2,5\}$ that
 \begin{align}\label{eqn_moreover1} F_{h,k}(-i^{ \ell_m - 3} \zeta_{c_m k}^h,-i^{3-\ell_m} \zeta_{a_m k}^{d_m h}) + F_{h,k}(i^{\ell_m-3} \zeta_{c_m k}^{h},i^{3-\ell_m} \zeta_{a_m k}^{d_m h}) = 0.
 \end{align}

\noindent \rm{(ii)} Let $\frac{h}{k} \in S_{41} \setminus \{-1\}$. Then we have that
 \begin{align}\label{eqn_cor2}
  \frac{-i}{24}&  \int_1^{i\infty}     \frac{E_4(z/288)}{\sqrt{-i\left(z+\tfrac{h}{k}\right)}}dz  =
 -\zeta_{288 k}^{11h}  F_{h,k}(\zeta_{24k}^h,\zeta_{24k}^{11h})  -\zeta_{288 k}^{35h} F_{h,k}( \zeta_{24k}^{5h},\zeta_{24k}^{7h}) \\ \notag
&\hspace{1.5in}+ \zeta_8^{-1}  \left(\tfrac{h}{k} +1\right)^{-\frac12}
  \left(\zeta_{288 K_1}^{11H_1}  F_{H_1,K_1}(\zeta_{24K_1}^{H_1},\zeta_{24K_1}^{11H_1})  +\zeta_{288 K_1}^{35H_1} F_{H_1,K_1}( \zeta_{24K_1}^{5H_1},\zeta_{24K_1}^{7H_1})\right).
  \end{align}
\noindent Moreover, we have that
\begin{align}\label{eqn_moreover2}
F_{h,k}(\zeta_{24k}^h,\zeta_{24k}^{11h}) + F_{h,k}(\zeta_{24k}^{5h},\zeta_{24k}^{7h})  + F_{h,k}(-\zeta_{24k}^h,-\zeta_{24k}^{11h}) + F_{h,k}(-\zeta_{24k}^{5h},-\zeta_{24k}^{7h}) = 0.
\end{align}
 
\end{corollary}
\begin{remark} The analogous result to (\ref{eqn_moreover1}) also holds for $m\in\{3,6\},$ however, the identity for these $m$ is trivial. \end{remark}
We illustrate Corollary \ref{cor_hypergeometric} in the following example.
\medskip \\  {\bf{Example.}}  Let $h/k=1/3 \in S_{11} \setminus \{\frac{-1}{2}\}$.  By Corollary \ref{cor_hypergeometric}, the Eichler integral of the eta-quotient $E_1$ appearing in (\ref{eqn_cor1}) may be evaluated exactly as
\begin{align*}\frac{-i}{8}   \int_{1/2}^{i\infty} \frac{E_1(z/32) dz}{\sqrt{-i(z+\tfrac13)}} &= \zeta^{-7}_{32} \sum_{n=0}^2 \frac{(-\zeta_6;\zeta_6)_n \zeta_{12}^{n(n+1)}}{(i \zeta_{24};\zeta_6)_{n+1}(-i \zeta_8;\zeta_6)_{n+1}} - \left(\tfrac35\right)^{\frac12} \zeta_{160}^{-37} \sum_{n=0}^4 \frac{(-\zeta_{10};\zeta_{10})_n \zeta_{20}^{n(n+1)}}{(i\zeta_{40};\zeta_{10})_{n+1}(-i\zeta_{40}^3;\zeta_{10})_{n+1}} \\ & \approx .05461 +.00825i.\end{align*}   Moreover, we have the following curious algebraic identity from (\ref{eqn_moreover1}):  \begin{align}\label{eqn_tranex}
 \sum_{n=0}^{2} \frac{(-\zeta_{6};\zeta_{6})_n \zeta_{12}^{n(n+1)}}{(i\zeta_{24};\zeta_{6})_{n+1} (-i\zeta_{8};\zeta_{6})_{n+1}}   + \sum_{n=0}^{2} \frac{(-\zeta_{6};\zeta_{6})_n \zeta_{12}^{n(n+1)}}{(-i\zeta_{24};\zeta_{6})_{n+1} (i\zeta_{8};\zeta_{6})_{n+1}}  = 0.\end{align}
 While (\ref{eqn_tranex}) may appear elementary, we point out that term by term, the two sums appearing are quite different. That is, let
 \begin{align*} a(n)  := \frac{(-\zeta_{6};\zeta_{6})_n \zeta_{12}^{n(n+1)}}{(i\zeta_{24};\zeta_{6})_{n+1} (-i\zeta_{8};\zeta_{6})_{n+1}}, \ \ \
 b(n) := \frac{(-\zeta_{6};\zeta_{6})_n \zeta_{12}^{n(n+1)}}{(-i\zeta_{24};\zeta_{6})_{n+1} (i\zeta_{8};\zeta_{6})_{n+1}}.
 \end{align*}  The following table gives the values of each summand appearing in (\ref{eqn_tranex});  other than the fact that the last summands satisfy $a(2) = -b(2)$, term-by-term cancellation in (\ref{eqn_tranex}) is not apparent.  Other examples which we numerically computed behaved similarly.
 $$\begin{array}{|c|c|c|} \hline n & a(n) & b(n)  \\ \hline
 0 & \approx 0.713123 -0.411722 i & \approx 0.384953  -0.222253 i \\ \hline
 1 & \approx -2.38616+1.37765 i & \approx 1.28808 -0.743673 i \\ \hline
 2 & \approx 1.22474 & \approx -1.22474 \\ \hline
 \end{array}$$

 Given that our quantum modular forms $V_{mn}$ satisfy the stronger property that their appropriate transformation properties hold on both a subset of $\mathbb Q$ and the upper half-plane $\mathbb H$, it is natural to ask if the functions $V_{mn}$ also extend into the lower half-plane $\mathbb H^-:=\{z\in \mathbb C \ | \ \text{Im}(z)<0\}$.  Indeed, in Section \ref{LOprelim}, we define for $m\in T$ the functions $\widetilde{E}_m(z)$ for $z\in \mathbb H^-$.  Upon making the change of variable $z=-2\tau/c_m^2$, where $\tau \in \mathbb H$ (and hence $z\in \mathbb H^-$), we show in Proposition \ref{prop_eichlere} that as $ \tau \to x \in S_{mn}\subseteq \mathbb Q$ from the upper half-plane, and hence as $z\to -2x/c_m^2\in \mathbb Q$ from the lower half-plane, the functions $\widetilde{E}_m\left(-{2x}/{c_m^2}\right)$ are quantum modular forms which transform exactly as our functions $V_{mn}(x)$ do in Theorem \ref{thm_quantum}, up to multiplication by a constant which can be explicitly determined.
\begin{proposition}  \label{prop_eichlere} For $m\in T$, the functions $\widetilde{E}_m$ are quantum modular forms of weight $1/2$.  In particular, for any $x\in S_{mn}$, up to multiplication by a constant, the functions $\widetilde{E}_m\left(-{2x}/{c_m^2}\right)$ satisfy the transformation laws given in Theorem \ref{thm_quantum} for the functions $V_{mn}(x)$.
\end{proposition}
Series similar to the functions $\widetilde{E}_m$ defined in (\ref{table_etil}) which instead arise from ordinary integer weight cusp forms were studied originally by Eichler (and are also often referred to as ``Eichler integrals"), and were shown to play fundamental roles within the theory of integer weight modular forms.   In the present setting, the modular objects $E_m$ related to the series $\widetilde{E}_m$ are not of integral weight, and many aspects of Eichler's theory become complicated. Nevertheless, in their fundamental work \cite{LZ}, Lawrence and Zagier  successfully consider Eichler's theory in the half integer weight setting; moreover, their work led to some of the first examples of quantum modular forms.  The functions $\widetilde{E}_m$ may also be viewed as \emph{partial theta functions}, which as series are similar to ordinary modular theta functions, but which are not modular in general \cite{apartial}.  Related results on quantum modular forms similar to those given in Proposition \ref{prop_eichlere} may be found in \cite{BR, FKTY, Zagier} among other places;  we follow their methods to prove Proposition \ref{prop_eichlere}.

\section{Preliminaries for Theorem \ref{thm_modvij} and Theorem \ref{thm_quantum}}
In this section we review previous work of Lemke Oliver \cite{LO},  Zwegers \cite{zwegers}, and the third author   \cite{Kang}, and make some preparations for our proofs of Theorems \ref{thm_modvij} and \ref{thm_quantum}.

\subsection{Work of Lemke Oliver on eta-theta functions}\label{LOprelim}
 We begin with the Dedkind eta-function (\ref{eqn_etadef}), whose well-known weight $1/2$ modular transformation properties are summarized in the following Lemma.
\begin{lemma} \label{lem_etatransf}
 For all
$\gamma=\sm{a}{b}{c}{d} \in \textnormal{SL}_2(\mathbb Z)$ and $\tau \in \mathbb H$, we have that
\begin{equation}
\eta\left(\gamma\tau\right)  = \psi\left(\gamma\right)(c\tau + d)^{\frac{1}{2}} \eta(\tau), \label{etatrgen}
\end{equation}
where $\psi\left(\gamma\right)$ is a $24$th root of unity, which can be given explicitly in terms of Dedekind sums \cite{Rad}.  In particular, we have that
\begin{equation*}
\eta\left( \frac{-1}{\tau} \right)= \sqrt{-i \tau} \, \eta(\tau).
\end{equation*}
\end{lemma}

In recent work, Lemke Oliver \cite{LO} proves that there are only eighteen weight $1/2$ eta-quotients that are also theta functions  or linear combinations of theta functions (with even character)  including twists by certain  principal characters.  Of these, thirteen are unique up to a change of variable $\tau \mapsto k\tau$; we list these as\footnote{Our ordering here differs from Lemke Oliver's.}

\begin{align}\label{e}
\begin{array}{lclclcl}
\ds e_1(\tau) &=& \ds \frac{\eta(\tau)^2}{\eta(2\tau)}  &=& \ds \sum_{n\in\Z} \left(1 - 2\leg{n}{2}^2  \right)q^{n^2}  \  = \ds \sum_{n\in\Z} (-1)^n q^{n^2},   \\
\ds e_2(\tau) &=& \ds \frac{\eta(2\tau)^5}{\eta(\tau)^2 \eta(4\tau)^2 }  &= &\ds \sum_{n\in\Z} q^{n^2},   \\
\ds e_3(\tau) &=&  \eta(24\tau)  &=& \ds \sum_{n\geq 1} \leg{12}{n} q^{n^2},\\
\ds e_4(\tau) &= & \ds\frac{\eta(48\tau) \eta(72\tau)^2 }{\eta(24\tau)\eta(144\tau)} & =  &\ds\sum_{n\geq 1} \leg{n}{6}^2 q^{n^2},  \\
\ds e_5(\tau) &= & \ds\frac{\eta(8\tau)\eta(32\tau) }{\eta(16\tau)} & = & \ds\sum_{n\geq 1} \leg{2}{n} q^{n^2},   \\
\ds e_6(\tau) &= &\ds \frac{\eta(16\tau)^ 2}{\eta(8\tau)} & = &\ds \sum_{n\geq 1} \leg{n}{2}^2 q^{n^2},   \\
\ds e_{7}(\tau) &=  &\ds\frac{\eta(3\tau) \eta(18\tau)^2 }{\eta(6\tau) \eta(9\tau)} & = &\ds \sum_{n\geq 1} \left(2\leg{n}{6}^2 - \leg{n}{3}^2 \right) q^{n^2},   \\
\ds e_8(\tau) &= &\ds \frac{\eta(6\tau)^2 \eta(9\tau) \eta(36\tau)}{\eta(3\tau)\eta(12\tau)\eta(18\tau)}  &=  &\ds\sum_{n\geq 1} \leg{n}{3}^2 q^{n^2},  \\
\ds e_9(\tau) &=&  \ds\frac{\eta(48\tau)^3}{\eta(24\tau) \eta(96\tau)}&  =  &\ds\sum_{n\geq 1} \leg{24}{n} q^{n^2},  \\
\ds e_{10}(\tau) &= & \ds\frac{\eta(24\tau) \eta(96\tau)\eta(144\tau)^5  }{\eta(48\tau)^2 \eta(72\tau)^2\eta(288\tau)^2} & = & \ds\sum_{n\geq 1} \leg{18}{n} q^{n^2}, \\
\ds e_{11}(\tau) &= & \ds\frac{\eta(\tau)\eta(4\tau) \eta(6\tau)^2}{\eta(2\tau)\eta(3\tau)\eta(12\tau) } & =&  \ds\sum_{n\in\Z} \left(1 - \frac{3}{2}\leg{n}{3}^2 \right) q^{n^2},  \\
\ds e_{12}(\tau) &=&  \ds\frac{\eta(2\tau)^2 \eta(3\tau)}{\eta(\tau)\eta(6\tau)}  &=  &\ds\sum_{n\in\Z} \left(1 - 2\leg{n}{2}^2 - \frac{3}{2}\leg{n}{3}^2 + 3\leg{n}{6}^2 \right) q^{n^2},   \\
\ds e_{13}(\tau) &= &\ds \frac{\eta(8\tau)^2 \eta(48\tau)}{\eta(16\tau) \eta(24\tau)} &= &\ds \sum_{n\geq 1} \left(3\leg{n}{6}^2 - 2\leg{n}{2}^2 \right) q^{n^2}.
\end{array}
\end{align}

Lemke Oliver  establishes a similar list for eta-quotients of weight $3/2$  (with odd character)  that are  theta functions or linear combinations of theta functions.  He finds the following six functions:

\begin{equation}\label{E}
\begin{tabular}{llcllllcll}
$\ds E_1(\tau)$&=&$\ds \eta(8\tau)^3$ &=&$\ds \sum_{n\geq 1}\leg{-4}{n}nq^{n^2} $,
&$\ds E_4(\tau)$&=&$\ds \frac{\eta(48\tau)^{13}}{\eta(24\tau)^5\eta(96\tau)^5}$&=&$\ds \sum_{n\geq 1}\leg{-6}{n} nq^{n^2} $,  \\
$\ds E_2(\tau)$&=&$\ds \frac{\eta(16\tau)^{9}}{\eta(8\tau)^3\eta(32\tau)^3}$&=&$\ds \sum_{n\geq 1}\leg{-2}{n} nq^{n^2} $,
&$\ds E_5(\tau)$&=&$\ds \frac{\eta(24\tau)^5}{\eta(48\tau)^2}$ &=& $\ds \sum_{n\geq 1}\leg{n}{12} nq^{n^2} $,  \\
$\ds E_3(\tau)$&=&$\ds \frac{\eta(3\tau)^2\eta(12\tau)^2}{\eta(6\tau)} $&=& $\ds \sum_{n\geq 1}\leg{n}{3} nq^{n^2} $,
&$\ds E_6(\tau)$&=&$\ds \frac{\eta(6\tau)^5}{\eta(3\tau)^2}$&=&$\ds  \sum_{n\geq 1}\left(2\leg{n}{12}-\leg{n}{3}\right)nq^{n^2} $.
\end{tabular}
\end{equation}

We also define the following functions for $z\in \mathbb H^-$,
\begin{align} \label{table_etil}
\begin{tabular}{ll}
$\displaystyle{\widetilde{E}_1(z) =\sum_{n\geq 1}\leg{-4}{n} e^{-2\pi i z n^2}}$,
&$\displaystyle{\widetilde{E}_4(z) = \sum_{n\geq 1}\leg{-6}{n}  e^{-2\pi i z n^2}}$,  \\
$\displaystyle{\widetilde{E}_2(z) = \sum_{n\geq 1}\leg{-2}{n}  e^{-2\pi i z n^2}}$,
&$\displaystyle{\widetilde{E}_5(z) = \sum_{n\geq 1}\leg{n}{12}  e^{-2\pi i z n^2}}$,  \\
$\displaystyle{\widetilde{E}_3(z) = \sum_{n\geq 1}\leg{n}{3}  e^{-2\pi i z n^2}}$,
&$\displaystyle{\widetilde{E}_6(z) =\sum_{n\geq 1}\left(2\leg{n}{12}-\leg{n}{3}\right) e^{-2\pi i z n^2}}$.
\end{tabular}
\end{align}
Although these functions are not modular forms, as series, their relationship to the modular eta-theta functions $E_m$ is apparent.  As discussed in Section \ref{intro}, these functions may be viewed as formal  Eichler integrals  of the modular eta-theta functions $E_m$, or, as  partial theta functions.  Connections between these types of functions and mock modular and quantum modular forms have been explored in a number of works, including \cite{BR, FKTY, FOR, Hikami, ZagierV, Zagier}.

\subsection{Work of Zwegers on mock theta functions related to unary theta functions}

Zwegers \cite{zwegers} provides a mechanism for constructing mock theta functions with shadow related to a given unary theta function of weight $3/2$.   These mock theta functions (which we discuss further in this context in Section  \ref{sub:construction}) feature the weight 1/2 theta functions
\[
\vartheta(v;\tau) := \sum_{n\in\Z}e^{2\pi i(n+\frac{1}{2})(v+\frac{1}{2})}q^{\frac{1}{2}(n+\frac{1}{2})^2}.
\]  It is well-known that these theta functions may be written as  
\begin{equation}\label{eq:theta}
\vartheta(v;\tau) = -iq^{\frac{1}{8}}e^{-\pi i v}\prod_{n\geq 1}(1-q^n)(1-e^{2\pi i v}q^{n-1})(1-e^{-2\pi i v}q^n).
\end{equation}
We note that $\vartheta(v;\tau)$ also satisfies the explicit modularity properties described in the following lemma.

\begin{lemma}[\cite{Rad}, (80.31) and (80.8)] \label{lem_thetatransf}
 For $\lambda,\mu \in \mathbb Z$,
$\gamma=\sm{a}{b}{c}{d} \in \textnormal{SL}_2(\mathbb Z),$ $z\in \mathbb C,$ and $\tau \in   \mathbb H$, we have that
\begin{align}
 \vartheta(z+\lambda \tau+\mu;\tau)&= (-1)^{\lambda+\mu}q^{-\frac{\lambda^2}{2}}e^{-2\pi i\lambda
z}\vartheta(z;\tau),\label{tt1}\\
\vartheta\left(\frac{z}{c\tau+d};
\gamma\tau\right)&= \psi^3\left(\gamma\right) (c\tau+d)^{\frac12}e^{\frac{\pi icz^2}{c \tau+d}}\vartheta(z;
\tau)\label{tt2}.
\end{align}
  In particular, we have that
\begin{equation*}
\vartheta\left(\frac{z}{\tau}; -\frac{1}{\tau} \right)= - i \sqrt{-i \tau}  e^{\frac{\pi i z^2}{\tau}}  \vartheta\left(z; \tau \right).
\end{equation*}
\end{lemma}

Now for $\tau\in\H$ and $u,v\in \C\backslash (\Z\tau + \Z)$,   Zwegers defines
\begin{equation}\label{series}
\mu(u,v;\tau) := \frac{e^{\pi i u}}{\vartheta(v;\tau)}\sum_{n\in\Z}\frac{(-1)^n e^{2\pi i n v}q^{\frac{n(n+1)}{2}}}{1 - e^{2\pi i u}q^{n}}.
\end{equation}
Zwegers also defines for $u\in\C$ and $\tau\in\mathbb{H}$ the Mordell integral $h$ by
\begin{equation}\label{Mordell}
h(u)=h(u;\tau):=\int_{\R}\frac{e^{\pi i\tau x^2 - 2\pi ux}}{\cosh{\pi x}}dx.
\end{equation}
We will make use of the following properties of $\mu$.
\begin{lemma}[Zwegers, Prop. 1.4 and 1.5 of \cite{zwegers}]  \label{lem_mu} Let $\mu(u,v):=\mu(u,v;\tau)$ and $h(u;\tau)$ be defined as in (\ref{series}) and (\ref{Mordell}). Then we have
\begin{enumerate}
\item $\mu(u+1,v)=-\mu(u,v)$,
\item $\mu(u,v+1)=-\mu(u,v)$,
\item $\mu(-u,-v)=\mu(u,v)$,
\item $\mu(u+z,v+z)-\mu(u,v)=\frac{1}{2\pi i}\frac{\vartheta '(0)\vartheta(u+v+z)\vartheta(z)}{\vartheta(u)\vartheta(v)\vartheta(u+z)\vartheta(v+z)}$, for $u,v,u+z,v+z\notin \Z\tau+\Z$,
\end{enumerate}
and the modular transformation properties,
\noindent
\begin{enumerate}
\item[(5)] $\mu(u,v;\tau+1)=e^{-\frac{\pi i}{4}}\mu(u,v;\tau)$,
\item[(6)] $\frac{1}{\sqrt{-i\tau}}e^{\pi i(u-v)^2/\tau}\mu\left(\frac{u}{\tau},\frac{v}{\tau};-\frac{1}{\tau}\right)+\mu(u,v;\tau)=\frac{1}{2i}h(u-v;\tau)$.
\end{enumerate}
\end{lemma}

Additionally, we will use the following theorem of  the third author\footnote{We have rewritten this formula to account for our definition of $\vartheta$ and $\mu$, which differs from \cite{Kang}.  In particular, writing $\vartheta_K$ and $\mu_K$ to indicate the notation in \cite{Kang},  we have that $\vartheta = -i \vartheta_K$ and $\mu = i \mu_K$.} \cite{Kang}, relating a certain specialization of $\mu(u,v;\tau)$ to a universal mock theta function.
\begin{theorem}[Kang \cite{Kang}] \label{kangthm}
If $\alpha\in\mathbb{C}$ such that $\alpha \not \in \frac12 \mathbb Z \tau + \frac12 \mathbb Z$, then
\[
\mu\left(2\alpha, \frac{\tau}{2};\tau\right) = iq^\frac18 g_2(e(\alpha);q^\frac12) -e(-\alpha)q^\frac18 \frac{\eta(\tau)^4}{\eta(\frac{\tau}{2})^2 \vartheta(2\alpha;\tau)},
\]
where $g_2$ is the universal mock theta function defined by
\[
g_2(z;q):=\sum_{n=0}^\infty \frac{(-q)_n q^{n(n+1)/2}}{(z;q)_{n+1}(z^{-1}q;q)_{n+1}}.
\]
\end{theorem}

The function $\mu$ is completed by defining the real-analytic function
\[
R(u;\tau):= \sum_{\nu \in\frac12 + \Z} \left({\rm{sgn}}(\nu) -  2\int_{0}^{(\nu +a)\sqrt{2y}} e^{-\pi t^2} dt  \right)(-1)^{\nu-\frac12}e^{-\pi i \nu^2 \tau - 2\pi i \nu u},
\]
with $y=\rm{Im}(\tau)$ and $a=\frac{\rm{Im}(u)}{\rm{Im}(\tau)}$. For $\tau\in\H$ and $u,v\in \C\backslash (\Z\tau + \Z)$,   Zwegers defines
\begin{equation}\label{eq:mu_tilde}
\widehat{\mu}(u,v;\tau) := \mu(u,v;\tau) + \frac{i}{2}R(u-v;\tau).
\end{equation}
The following explicit transformation properties show that $\widehat{\mu}$ transforms like a two-variable  (non-holomorphic) Jacobi form of weight $1/2$.

\begin{lemma}[Zwegers, Prop. 1.11(1,2) of \cite{zwegers}]\label{lem_muhat} Let $\widehat{\mu}(u,v;\tau)$ be defined as in (\ref{eq:mu_tilde}). Then
\begin{enumerate}
\item $\widehat{\mu}(u+k\tau+l,v+m\tau+n;\tau)=(-1)^{k+l+m+n}e^{\pi i(k-m)^2\tau+2\pi i(k-m)(u-v)}\widehat{\mu}(u,v;\tau)$, for $k,l,m,n\in\Z$, and 
\item $\widehat{\mu}\left(\frac{u}{c\tau+d},\frac{v}{c\tau+d};\frac{a\tau+b}{c\tau+d}\right)=v(\gamma)^{-3}(c\tau+d)^{\frac12}e^{-\pi ic(u-v)^2/(c\tau+d)}\widehat{\mu}(u,v;\tau)$, for $\gamma=\sm{a}{b}{c}{d}\in\textnormal{SL}_2(\mathbb Z)$, with $v(\gamma)$ defined as in (\ref{etatrgen}).
\end{enumerate}
\end{lemma}

As we shall see in Theorem \ref{thm_z116}, these completions are related to the unary theta function defined for $a,b\in\R$ and $\tau\in\H$ by
\begin{equation}\label{g}
g_{a,b}(\tau) := \sum_{n\in\Z} (n+a)e^{2\pi i b (n+a)}q^\frac{(n+a)^2}{2}.
\end{equation}
The following transformation properties show, in particular, that $g_{a,b}$ is a modular form of weight $3/2$ when $a$ and $b$ are rational. 
\begin{lemma}[Zwegers, Prop. 1.15 of \cite{zwegers}]\label{lem_gabprop} The function $g_{a,b}$ satisfies the following:
\begin{enumerate}
\item $g_{a+1,b}(\tau)=g_{a,b}(\tau)$,
\item $g_{a,b+1}(\tau)=e^{2\pi ia}g_{a,b}(\tau)$,
\item $g_{-a,-b}(\tau)=-g_{a,b}(\tau)$,
\item $g_{a,b}(\tau +1)=e^{-\pi ia(a+1)}g_{a,a+b+\frac12}(\tau)$,
\item $g_{a,b}(-\frac{1}{\tau})=ie^{2\pi iab}(-i\tau)^{3/2}g_{b,-a}(\tau)$.
\end{enumerate}
\end{lemma}

The unary theta function $g_{a,b}$ is related to both $R$ and $h$ by the following theorem.

\begin{theorem}[Zwegers, Thm. 1.16 of \cite{zwegers}]\label{thm_z116}
For $\tau \in \H$, we have the following two results.

When $a\in (-\frac12, \frac12)$ and $b \in \R$,
\begin{equation}\label{integral}
\int_{-\overline{\tau}}^{i\infty} \frac{g_{a+\frac{1}{2},b+\frac{1}{2}}(z)}{\sqrt{-i(z+\tau)}}dz = -e^{2\pi i a(b+\frac{1}{2})}q^{-\frac{a^2}{2}} R(a\tau-b;\tau).
\end{equation}

Also, when $a,b \in (-\frac12, \frac12)$,
\begin{equation}\label{integral_h}
\int_{0}^{i\infty} \frac{g_{a+\frac{1}{2},b+\frac{1}{2}}(z)}{\sqrt{-i(z+\tau)}}dz =-e^{2\pi i a(b+\frac{1}{2})}q^{-\frac{a^2}{2}} h(a\tau-b;\tau).
\end{equation}
\end{theorem}

We extend Theorem \ref{thm_z116} in the following result, which we will use in our proof of Theorem \ref{thm_quantum}.

  \begin{lemma}\label{lem_Zlem}  Let $\tau \in \mathbb H$.  \begin{itemize}\item[i)] For $b\in\mathbb R\setminus \frac12\mathbb Z$, $$\int_{-\overline{\tau}}^{i\infty} \frac{g_{1,b+\frac12}(z)}{\sqrt{-i(z+\tau)}} dz = -ie\left(-\frac{\tau}{8} + \frac{b}{2}  \right)R\left(\frac{\tau}{2}-b;\tau\right) + i.$$
  \item[ii)] For $b\in (-\frac12,\frac12)\setminus\{0\}$, $$\int_{0}^{i\infty} \frac{g_{1,b+\frac12}(z)}{\sqrt{-i(z+\tau)}}dz = -ie\left(-\frac{\tau}{8} + \frac{b}{2}  \right)h\left(\frac{\tau}{2}-b;\tau\right) + i.$$
   \item[iii)] For $a\in (-\frac12,\frac12)\setminus\{0\}$, $$\int_{0}^{i\infty} \frac{g_{a+1/2,1}(z)}{\sqrt{-i(z+\tau)}}dz = -e\left(-\frac{a^2}{2}\tau + a  \right)h\left(a\tau-\frac12;\tau\right) + \frac{e(a)}{\sqrt{-i\tau}}.$$
  \end{itemize}
  \end{lemma}

  \begin{proof}[Proof of Lemma \ref{lem_Zlem}]  If $b\in \mathbb R \setminus \frac12 \mathbb Z$, we have that $g_{1,b+\frac12}(z) = O\left(e^{-\pi \textnormal{Im}(z)}\right).$  If $a\in(-1/2,1/2)\setminus \{0\}$, we have that $g_{a+\frac12,1}(z) = O\left(e^{-\pi v_0^2  \textnormal{Im}(z)}\right)$ for some $v_0>0$ as $\textnormal{Im}(z) \to \infty$.  These facts justify the convergence of the integrals in Lemma \ref{lem_Zlem}.

  The proof of (\ref{integral}) in \cite{zwegers} yields that the integral on the left hand side of i) in Lemma \ref{lem_Zlem} equals
  \begin{align} \label{eqn_g1R} -ie\left(-\frac{\tau}{8}+ \frac{b}{2}\right) \sum_{v\in \frac12 + \mathbb Z} \left(\text{sgn}\left(v+\frac12\right) - 2\int_0^{\left(v+\frac12\right)\sqrt{2\textnormal{Im}(\tau)}} e^{-\pi u^2} du \right)(-1)^{v-\frac12}e\left(-\frac{v^2\tau}{2} - v\left(\frac{\tau}{2}-b\right)\right).\end{align}
 Now for all $v\in (\frac12 + \mathbb Z) \setminus \{ -\frac12\}$, we have that $\text{sgn}\left(v+\frac12\right) = \text{sgn}(v)$.  For $v=-\frac12$, we have that $0=\text{sgn}\left(-\frac12 + \frac12\right)   = \text{sgn}\left(-\frac12\right) + 1$. Making these substitutions into (\ref{eqn_g1R}) and simplifying proves part i) of Lemma \ref{lem_Zlem}.

 Part ii) and part iii) of Lemma \ref{lem_Zlem} now follow as argued in Remark 1.20 in \cite{zwegers}  by using part i) of Lemma \ref{lem_Zlem} above (rather than (\ref{integral})) where necessary.
  \end{proof}

In the following section, we review the connection between $\widehat{\mu}$ and the theory of harmonic Maass forms.

\subsection{Harmonic Maass forms of weight $1/2$, and period and Mordell integrals}\label{sub:construction}
 Following Bruinier and Funke \cite{BF}, a \emph{harmonic Maass form} $\widehat{f}:\mathbb{H}\to\mathbb{C}$ is a non-holomorphic extension of  a classical modular form.  It is a smooth function 
 such that for  a  weight $\kappa \in \frac{1}{2}\Z$, if $\Gamma \subseteq \rm{SL}_2(\Z)$ and $\chi$ is a Dirichlet character modulo $N$, then for all $\gamma=\left(\begin{smallmatrix}a & b\\c & d\end{smallmatrix}\right) \in \Gamma$ and $\tau \in \H$ we have $\widehat{f}(\gamma \tau) = \chi(d)(c\tau+d)^{\kappa}\widehat{f}(\tau)$. 
Moreover, $\widehat{f}$ must vanish under the weight $\kappa$ Laplacian operator defined, for $\tau = x+iy$, by
\[
\Delta_{\kappa}:=-y^2\left(\frac{\partial^2}{\partial x^2}+\frac{\partial^2}{\partial y^2}\right)+i\kappa y\left(\frac{\partial}{\partial x}+i\frac{\partial}{\partial y}\right).
\]
Additionally, $\widehat{f}$ must have at most linear exponential growth at all cusps.

The Fourier series of a harmonic  Maass form $\widehat{f}$ of weight $\kappa$ naturally decomposes as the sum of  a holomorphic and a non-holomorphic part.  We refer to the holomorphic part $f$  as a \emph{mock modular form} of weight $\kappa$  after Zagier \cite{ZagierB}.    In the special  case $\kappa\in \{1/2,3/2\}$,   we refer to  $f$ as a \emph{mock theta function}. Moreover, a harmonic Maass form $\widehat{f}$ of weight $\kappa$ is mapped to a classical modular form of weight $2-\kappa$ by the differential operator
\[
\xi_{\kappa}:=2iy^{\kappa}\cdot\overline{\frac{\partial}{\partial{\overline{\tau}}}}.
\]
The image of $\widehat{f}$ under $\xi_{\kappa}$ is called the \emph{shadow} of $ f  $. We next show that certain  specializations of the function  $\mu$ are essentially  mock theta functions with shadows related to $g_{a,b}$. Similar results are known, however in this paper we require and thus establish the precise statement given in  Proposition 2.8.  To state it, we define for a function $g:\mathbb H\to \mathbb C$ its \emph{complement} $$g^c(\tau) := \overline{g(-\overline{\tau})}.$$  For $\tau \in \mathbb H$, we define for $a,b\in\mathbb R$ and $u,v \in \mathbb C \setminus (\mathbb Z \tau + \mathbb Z)$   the function
\begin{align}\label{def_Mhat}\widehat{M}_{a,b}(\tau):=-\sqrt{2} e^{2\pi i a\left(b+\frac12\right)} q^{-\frac{a^2}{2}} \widehat{\mu}(u,v;\tau).\end{align}
We denote the holomorphic part of $\widehat{M}_{a,b}$ by $M_{a,b}$, that is,
$ M_{a,b}(\tau):=-\sqrt{2} e^{2\pi i a\left(b+\frac12\right)} q^{-\frac{a^2}{2}}  \mu(u,v;\tau).$

\begin{proposition}\label{prop_p1shadows}  Let $\tau \in \mathbb H$, and $u,v \in \mathbb C \setminus (\mathbb Z\tau + \mathbb Z).  $  If $u-v=a\tau - b$ for some $a,b \in \mathbb R$, then  the function $\widehat{M}_{a,b}(\tau)$ satisfies
\begin{enumerate} \item[(i)] $\xi_{\frac12}\left(\widehat{M}_{a,b}(\tau)\right) = g^c_{a+\frac12,b+\frac12}(\tau), $
\item[(ii)] $\Delta_{\frac12} (\widehat{M}_{a,b}(\tau)) = 0.$
\end{enumerate}
  \end{proposition}

\begin{remark} Part (ii) of Proposition \ref{prop_p1shadows} together with the transformation laws established in Lemma \ref{lem_muhat} show that $\widehat{M}_{a,b}$ is essentially a harmonic Maass form of weight $1/2$ for suitable $v, a, $ and $b$; we illustrate this more precisely in the proof of Theorem \ref{thm_modvij}.      \end{remark}
  \begin{proof} 
  Here and throughout, we write $\tau = x + i y$.    We begin by establishing part (i).  We have that
  \begin{align}\notag \xi_{\frac12} \left(\widehat{M}_{a,b}(\tau)\right) &= \xi_{\frac12} \left( -\sqrt{2} e^{2\pi i a\left(b+\frac12\right)} q^{-\frac{a^2}{2}}  \mu (u,v;\tau) -\sqrt{2} e^{2\pi i a\left(b+\frac12\right)} q^{-\frac{a^2}{2}} \cdot \frac{i}{2} R(a\tau-b;\tau) \right) \\
  \notag &= 2 i y^{\frac 12} \overline{\frac{\partial}{\partial \overline{\tau}}\left(  -\sqrt{2} e^{2\pi i a\left(b+\frac12\right)} q^{-\frac{a^2}{2}} \cdot \frac{i}{2} R(a\tau-b;\tau)\right) }
 \\ \label{eqn_shad1}
  &= - \sqrt{2} y^{\frac12}e^{-2\pi i a \left(b+\frac12\right)} \overline{q^{-\frac{a^2}{2}} \frac{\partial}{\partial \overline{\tau}}    R(a\tau-b;\tau).}
  \end{align}
   It is shown in \cite[(1.5)]{zwegers} that
\begin{align}  \frac{\partial}{\partial \overline{\tau}}R (a\tau - b;\tau)    &= -\frac{i}{\sqrt{2 y}} e^{-2\pi a^2 y}\sum_{n\in \mathbb Z}(-1)^n \left(n+a+ \tfrac12\right) e^{-\pi i \left(n+  \frac12\right)^2x}e^{-\pi \left(n +\frac12\right)^2y}e^{ - 2\pi i \left(n+\frac12\right)\left(a x-b   \right)}e^{ - 2\pi   \left(n+\frac12\right)ay}. \label{eqn_xiR}
  \end{align}  Taking the conjugate of (\ref{eqn_xiR}) and $q^{-a^2/2}$, we find that (\ref{eqn_shad1}) becomes
  \begin{align*}  
  &\hspace{.1in}= -ie^{-2\pi i a \left(b+\frac12\right)} q^{\frac{a^2}{2}} \sum_{n\in \mathbb Z}(-1)^n \left(n+a+ \tfrac12\right) q^{\frac12 \left(n+ \frac12\right)^2} q^{  \left(n+\frac12\right)a }e^{ - 2\pi  i  \left(n+\frac12\right)b} \\
  &\hspace{.1in}=      \sum_{n\in \mathbb Z}  \left(n+a+ \tfrac12\right) q^{\frac12 \left(n+a+ \frac12\right)^2}  e^{ - 2\pi  i  \left(n+a+\frac12\right)\left(b + \frac12\right)}  = g_{a+\frac12,-b-\frac12}(\tau).
  \end{align*}
  Using the definition of $g_{a+\frac12,-b-\frac12}(\tau)$, it is not difficult to show that
 \begin{align}\label{eqn_3g} g_{a+\frac12,-b-\frac12}(\tau) = \overline{g_{a+\frac12,b+\frac12}(-\overline{\tau})} =  g^c_{a+\frac12,b+\frac12}(\tau).\end{align}  This proves part (i).

To prove part (ii), we use the fact that the weight $1/2$ Laplacian operator factors as $\Delta_{\frac12} = - \xi_{\frac32} {\xi_{\frac12}}$.  The result follows by applying $-\xi_{\frac32}$ to the expression given in part (i) of the Proposition, using (\ref{eqn_3g}).  \end{proof}

\subsection{Converting setting of Lemke Oliver to notation of Zwegers}

We now express the eta-theta functions from (\ref{e}) and (\ref{E}) in terms of  the theta  functions $\vartheta(v;\tau)$ and $g_{a,b}(\tau)$. We first observe that we can convert the sums over positive integers in the definitions of the  functions $E_m$  into sums over all integers, and then write the $E_m$ in terms of   the functions  $g_{a,b}$. For example, we see from the definition of $g_{a,b}(\tau)$ in (\ref{g}) that
\[
4g_{\frac{1}{4},0}(32\tau) = 4\sum_{n\in\Z}\left(n+\frac{1}{4}\right)q^{16\left(n+\frac{1}{4}\right)^2} = \sum_{n\in\Z}\left(4n+1\right)q^{\left(4n+1\right)^2}=E_1(\tau).
\]
By similar methods we find the following result.
 
\begin{lemma}\label{E-g}  For $\tau \in \mathbb H$, we have that
\begin{align} 
E_1(\tau) &= 4g_{\frac{1}{4},0}(32\tau) =\sum_{n\in\Z}(4n+1)q^{(4n+1)^2},\\
E_2(\tau) &= 4e^{\frac{-\pi i}{4}}g_{\frac{1}{4},\frac{1}{2}}(32\tau) =\sum_{n\in\Z}(-1)^n(4n+1)q^{(4n+1)^2}, \nonumber\\
E_3(\tau) &= 3g_{\frac{1}{3},0}(18\tau)   =\sum_{n\in\Z}(3n+1)q^{(3n+1)^2},\nonumber\\
E_4(\tau) &= 12e^{\frac{-\pi i}{12}}g_{\frac{1}{12},\frac{1}{2}}(288\tau) + 12e^{\frac{-5\pi i}{12}}g_{\frac{5}{12},\frac{1}{2}}(288\tau) \nonumber \\
&= \sum_{n\in\Z}(-1)^n(12n+1)q^{(12n+1)^2}  + \sum_{n\in\Z}(-1)^n(12n+5)q^{(12n+5)^2},\nonumber \\
E_5(\tau) &= 6g_{\frac{1}{6},0}(72\tau) =\sum_{n\in\Z}(6n+1)q^{(6n+1)^2},\nonumber\\
E_6(\tau) &= 3e^{\frac{-\pi i}{3}}g_{\frac{1}{3},\frac{1}{2}}(18\tau) =\sum_{n\in\Z}(-1)^n(3n+1)q^{(3n+1)^2}.\nonumber
\end{align}
\end{lemma}

From Proposition \ref{prop_p1shadows}  and Lemma \ref{E-g},  we see that to construct forms $\widehat{M}_{a,b}$ whose   images under the $\xi_{\frac12}$-operator  
 are equal to a constant multiple of $E_m(\tau/k_m)$ for some suitable constants $k_m$
 we are only restricted by $u-v$ for $u,v\in \C\backslash (\Z\tau + \Z)$, not by $u$ and $v$ individually.  Since the theta function $\vartheta(v;\tau)$ appears as a prominent factor in the definition of $\widehat{\mu}$ from \eqref{eq:mu_tilde}, we again use the classification in \cite{LO} to restrict to those $\vartheta(v;\tau)$ which are eta-quotients of weight $1/2$ appearing in the list in (\ref{e}).

\begin{lemma}\label{thetas}   For $\tau \in \mathbb H$, we have that
\begin{align} 
\vartheta\left( \frac{\tau}{2} ; \tau\right) &= -iq^{-\frac{1}{8}}e_1\left(\frac{\tau}{2}\right), &
\vartheta\left( \frac{\tau}{4} ; \tau\right) &= -iq^{-\frac{1}{32}}e_5\left(\frac{\tau}{32}\right),  \\
\medskip
\vartheta\left( \frac{\tau}{2} - \frac{1}{2} ; \tau\right) &= q^{-\frac{1}{8}}e_2\left(\frac{\tau}{2}\right), \nonumber & \vartheta\left( \frac{\tau}{4} -\frac{1}{2} ; \tau\right) &= q^{-\frac{1}{32}}e_6\left(\frac{\tau}{32}\right),\\
\medskip
\vartheta\left( \frac{\tau}{3} ; \tau\right) &= -iq^{-\frac{1}{18}}e_3\left(\frac{\tau}{72}\right),\nonumber & \vartheta\left( \frac{\tau}{6} ; \tau\right) &= -iq^{-\frac{1}{72}}e_{7}\left(\frac{\tau}{18}\right), \nonumber \\
\medskip
\vartheta\left( \frac{\tau}{3} - \frac{1}{2} ; \tau\right) &= q^{-\frac{1}{18}}e_4\left(\frac{\tau}{72}\right), \nonumber & 
\vartheta\left( \frac{\tau}{6}  -\frac{1}{2} ; \tau\right) &= q^{-\frac{1}{72}}e_8\left(\frac{\tau}{18}\right). \nonumber
\medskip
\end{align}
\end{lemma}

\begin{proof}
Using \eqref{eq:theta}, we have that 
\begin{align*}
\vartheta\left(\tau-\frac{1}{2};2\tau\right) &= -iq^{\frac{1}{4}}e^{-\pi i (\tau - \frac{1}{2})}\prod_{n\geq 1}(1-q^{2n})(1-e^{2\pi i(\tau-\frac{1}{2})}q^{2n-2})(1-e^{-2\pi i(\tau-\frac{1}{2})}q^{2n})\\
&= q^{-\frac{1}{4}}\prod_{n\geq 1}(1-q^{2n})(1+q^{2n-1})^2 
= q^{-\frac{1}{4}}\prod_{n\geq 1}\frac{(1-q^{2n})(1+q^{n})^2}{(1+q^{2n})^2} \cdot \frac{(1-q^n)^2(1-q^{2n})^2}{(1-q^{2n})^2(1-q^n)^2} \\
&= q^{-\frac{1}{4}}\prod_{n\geq 1}\frac{(1-q^{2n})^5}{(1-q^n)^2(1-q^{4n})^2} = q^{-\frac{1}{4}} \frac{\eta(2\tau)^5}{\eta(\tau)^2\eta(4\tau)^2} =  q^{-\frac{1}{4}}e_2(\tau),\\
\end{align*}
which is the first identity above with $\tau \to \tau/2$.  The rest follow from similar arguments.
 \end{proof}

 Note that $\vartheta\left( \frac{\tau}{3} ; \tau - \frac{1}{2}\right) = e(-\frac{3}{8})q^{-\frac{1}{18}}e_{10}\left(\frac{\tau}{72}\right)$ and $\vartheta\left(\frac{\tau}{3} -1/2; \tau -\frac{1}{2}\right) = e(\frac{1}{8})q^{-\frac{1}{18}}e_{11}\left(\frac{\tau}{72}\right)$, which are not of the form $\vartheta(v; \tau)$.  Similarly, $e_{11}, e_{12}$, and $e_{13}$ cannot be written in the form $\vartheta(v;\tau)$.  Hence, we restrict our constructions to the first eight $e_n$ functions.

\section{Eta-theta functions and mock modular forms}\label{constructingV}

We are now ready to construct our families of mock theta functions.  For each weight $3/2$ theta function $E_m$ we construct eight corresponding functions $V_{mn}$, one for each weight $1/2$ theta function $e_{n}$.  However, in some cases the $V_{mn}$ are degenerate due the presence of poles.  Here, we will focus on the construction of $V_{11}$ as the other constructions follow similarly.   Our goal for $V_{11}$ is to construct a   function that has shadow associated to $E_1$, and the factor $e_1$ in its series representation.  

First, we observe from Lemma \ref{E-g} that $E_1(\tau) = 4g_{\frac{1}{4},0}(32\tau)$.  Thus, we make the change $\tau \mapsto 32\tau$ and consider a function of the form $\mu(u,v;32\tau)$ as in \eqref{series}.  We choose $v$ so that the theta function $\vartheta(v;32\tau)$ appearing in \eqref{series} is  in terms of   $e_1$.  By Lemma \ref{thetas} we see that we should choose $v= 16\tau$ so that  $\vartheta\left( 16\tau; 32\tau\right) = -iq^{-4}e_1\left(16\tau\right)$.  By Proposition \ref{prop_p1shadows} the corresponding function $M_{-\frac{1}4,-\frac{1}2}(32\tau)$ has shadow related to $g_{\frac14, 0}(32\tau)$, so long as
$$u-16\tau = u-v= -\frac{1}4(32\tau)-\left(-\tfrac{1}2\right) = -8\tau+\tfrac12.$$
Thus we choose $u= 8\tau+1/2$, and calculate the series form of $-q^{-1}\mu(8\tau+1/2, 16\tau; 32\tau)$ using \eqref{series}.  Our final step is to renormalize with $\tau \mapsto \tau/32$.  We obtain  
\[
V_{11}(\tau) = \frac{q^{-9/32}}{e_1(\tau/2)}\sum_{n\in\Z}\frac{(-1)^nq^{(n+1)^2/2}}{1+q^{n+1/4}}=w_1 q^{t_1} \mu \left( u^{(11)}_\tau,v^{(11)}_\tau;\tau \right) := -q^{-1/32}\mu\left(\frac{\tau}{4}+\frac{1}{2}, \frac{\tau}{2}; \tau\right).
\]

We repeat this process for each of the remaining $e_n$, to find $V_{1n}$ as above with shadow related to $E_1$ and with the theta function $e_n$ as a factor in its series representation.   We find that the construction of $V_{16}$ leads to the choice $u= 0 \in \C\backslash (32\Z\tau+\Z)$, and so this fails to produce a mock modular form due to the existence   of poles.    Each of the $V_{1n}$ are fully listed in the Appendix.

We repeat this entire process for each $E_m$ with $m\in\{2,3,5,6\}$.  The case $E_4$ requires some additional care as $E_4(\tau) = 12e^{\frac{-\pi i}{12}}g_{\frac{1}{12},\frac{1}{2}}(288\tau) + 12e^{\frac{-5\pi i}{12}}g_{\frac{5}{12},\frac{1}{2}}(288\tau)$.  In this case we build two different forms $V_{4'n}$ and $V_{4''n}$, one for each $g_{a,b} $ function, using the process described above.  We then add these forms to create our desired mock theta function.  For example,  \begin{align*}
V_{41}(\tau) = V_{4'1}(\tau) + V_{4''1}(\tau)
&=\frac{-q^{-121/288}}{e_1(\tau/2)}\sum_{n\in\Z}\frac{(-1)^nq^{(n+1)^2/2}}{1-q^{n+1/12}}
 +\frac{-q^{-49/288}}{e_1(\tau/2)}\sum_{n\in\Z}\frac{(-1)^nq^{(n+1)^2/2}}{1-q^{n+5/12}},
 \\&=iq^{-25/288}\mu\left(\frac{\tau}{12}, \frac{\tau}{2}; \tau\right)+iq^{-1/288}\mu\left(\frac{5\tau}{12}, \frac{\tau}{2}; \tau\right).
\end{align*}
All of the $V_{mn}$, including the $V_{4n}$, are listed in the Appendix.

\begin{remark}
Using Lemma \ref{lem_mu}, we see  \[V_{57}(\tau)=-q^{-1/18}\mu\left(-\frac{\tau}{6}+\frac{1}{2}, \frac{\tau}{6}; \tau\right) = -q^{-1/18}\mu\left(\frac{\tau}{6},-\frac{\tau}{6}+\frac{1}{2}; \tau\right) = -q^{-1/18}\mu\left(-\frac{\tau}{6},\frac{\tau}{6}-\frac{1}{2}; \tau\right) =V_{58}(\tau).\]
A comparison of coefficients reveals that all other series are unique.
\end{remark}

\begin{remark}
A coefficient search on the On-Line Encyclopedia of Integer Sequences suggests the following equalities: 
\[ -q^{1/24}V_{41}(12\tau) = \psi(q), \quad q^{1/24}V_{58}(3\tau)=\chi(q), \quad q^{-2/3}V_{64}(6\tau)=\rho(q),\]
\[ -q^{1/8}V_{21}(4\tau)=A(q), \quad -q^{1/8}V_{12}(4\tau) = U_1(q), \quad 2q^{1/8}V_{15}(4\tau)=U_0(q),\] where $\psi(q)$, $\chi(q)$, and $\rho(q)$ are Ramanujan's third order mock thetas, $A(q)$ is Ramanujan's second order mock theta, and
$U_1(q)$ and $U_0(q)$ are Gordon and McIntosh's eighth order mock thetas.  These series are defined in \cite{Ram} and \cite{GM}.  The latter three equalities follow from the definitions in \cite{GM} and Lemma \ref{lem_mu}.
\end{remark}

\subsection{Proof of Theorem \ref{thm_modvij}}
First, we wish to establish the mock modularity of the 51 functions $V_{mn}$ for admissible $(m,n)$ when $m\in T' \setminus \{4\}$. We will make use of Proposition \ref{prop_p1shadows}, but must also establish the modular transformation properties of these functions.   For such pairs $(m,n)$,  the functions $V_{mn}$,  as summarized in the Appendix, may be  expressed in terms of the $\mu$-function, and parameters $w_m, t_m, u^{(mn)}_\tau, v^{(mn)}_\tau$ as
\begin{align} \label{eq_Vnotation}
V_{mn}(\tau) &= w_m q^{t_m} \mu(u^{(mn)}_\tau,v^{(mn)}_\tau;\tau). \end{align} We denote their completions by 
\begin{align} \label{eq_Vnotation2}
\widehat{V}_{mn}(\tau) &:= w_m q^{t_m} \widehat{\mu}(u^{(mn)}_\tau,v^{(mn)}_\tau;\tau).
\end{align}  When $m=4$, for any admissible $n$, we consider
$$V_{4n}(\tau) = V_{4'n}(\tau) + V_{4''n}(\tau),$$ and the completions $$\widehat{V}_{4n}(\tau) :=  \widehat{V}_{4'n}(\tau) + \widehat{V}_{4''n}(\tau).$$   Toward establishing their modular properties,  we establish the following preliminary lemmas, the first of which follows directly from Lemma \ref{lem_muhat}.   Throughout, for  $x_{\tau} \in \mathbb C \setminus (\mathbb Z \tau +  \mathbb Z)$, we define for any $\gamma = \sm{a}{b}{c}{d} \in \textnormal{SL}_2(\mathbb Z)$,
\[
\widetilde{x}_{\gamma, \tau} := x_{\gamma \tau} \cdot (c\tau+d).
\]

 \begin{lemma} \label{lem_transhat}Let $\gamma = \sm{a}{b}{c}{d} \in \textnormal{SL}_2(\mathbb Z)$, $\tau \in \mathbb H$, and $ u_\tau,  v_\tau \in \mathbb C \setminus (\mathbb Z \tau +  \mathbb Z).$   Suppose $\widetilde{u}_{\gamma,\tau} = u_{\tau} + k_\gamma\cdot \tau + \ell_\gamma,$ and $\widetilde{v}_{\gamma,\tau} = v_{\tau} + r_\gamma\cdot \tau  + s_\gamma$, for some integers $k_\gamma, \ell_\gamma, r_\gamma, s_\gamma$.  Then we have that
\begin{align*} \widehat{\mu}&\left(u_{\gamma \tau}, v_{\gamma \tau}; \gamma \tau\right)  = \widehat{\mu}\left(\frac{\widetilde{u}_{\gamma,\tau}}{c\tau+d},\frac{\widetilde{v}_{\gamma,\tau}}{c\tau+d}; \gamma \tau\right) \\ &\hspace{.2in}= \psi(\gamma)^{-3}(-1)^{k_\gamma+\ell_\gamma+r_\gamma+s_\gamma}(c\tau+d)^{\frac12} q^{\frac{(k_\gamma-r_\gamma)^2}{2}} e\left(\frac{-c(\widetilde{u}_{\gamma,\tau}-\widetilde{v}_{\gamma,\tau})^2}{2(c\tau+d)}  + (k_\gamma-r_\gamma)(u_\tau-v_\tau)\right)\widehat{\mu}(u_\tau,v_\tau;\tau).\end{align*}
\end{lemma}
\qed \ \\

We next establish two technical lemmas, Lemma \ref{lem_invuv} and Lemma \ref{lem_integers}, which will allow us to efficiently establish the mock modularity of our functions $V_{mn}$, when combined with Lemma \ref{lem_transhat} and Proposition \ref{prop_p1shadows} above.  
\begin{lemma} \label{lem_invuv}   Let $\gamma = \sm{a}{b}{c}{d} \in \textnormal{SL}_2(\mathbb Z)$,  $\tau \in \mathbb H, j\in \{1,2\},$ and $u_{\tau}^{(j)}, v_\tau^{(j)} \in \mathbb C \setminus (\mathbb Z \tau +  \mathbb Z).$ Suppose there exist constants $ k_{\gamma}^{(j)},  \ell_{\gamma}^{(j)},   r_{\gamma}^{(j)},  s_{\gamma}^{(j)} \in \mathbb R$ satisfying $\widetilde{u}_{\gamma,\tau}^{(j)}  = u^{(j)}_{\tau} + k_{\gamma}^{(j)} \cdot \tau + \ell_{\gamma}^{(j)},$ and $\widetilde{v}_{\gamma,\tau}^{(j)}  = v^{(j)}_{\tau} + r_{\gamma}^{(j)}\cdot\tau  + s_{\gamma}^{(j)}$.   Further, define  the difference functions \begin{align*}& d_{\tau}^{(j)} := u^{(j)}_{\tau} - v^{(j)}_{\tau},  &  \widetilde{d}^{(j)}_{\gamma,\tau} &:= \widetilde{u}_{\gamma,\tau}^{(j)} - \widetilde{v}_{\gamma,\tau}^{(j)}, \\
& \delta^{(j)}_{\gamma} : =   k_{\gamma}^{(j)}  - r_{\gamma}^{(j)}, &\rho^{(j)}_{\gamma}& := \ell_{\gamma}^{(j)} - s^{(j)}_{\gamma}.\end{align*}
Suppose that $d^{(1)}_{\tau} = d^{(2)}_{\tau}$.  Then we have that \begin{align*}
\widetilde{d}^{(1)}_{\gamma,\tau} =  \widetilde{d}^{(2)}_{\gamma,\tau},  \ \  \ \
 {\delta}^{(1)}_{\gamma}  = \delta^{(2)}_\gamma, \ \ {\text{and \ \ }} \rho^{(1)}_\gamma  = \rho^{(2)}_\gamma. \end{align*}
\end{lemma}
\begin{proof}[Proof of Lemma \ref{lem_invuv}]
The first assertion follows from the fact that $d_\tau^{(1)} = d_\tau^{(2)},$ that $c\tau + d \neq 0$, and the definitions of $\widetilde{u}_{\gamma,\tau}^{(j)}$ and $\widetilde{v}_{\gamma,\tau}^{(j)}$.  To prove the second and third assertions,  we have  that $\widetilde{u}_{\gamma,\tau}^{(j)} = u_\tau^{(j)} + k_\gamma^{(j)}\cdot \tau  + \ell_\gamma^{(j)}$ and $\widetilde{v}_{\gamma,\tau}^{(j)} = v_\tau^{(j)} + r_\gamma^{(j)}\cdot \tau  + s_\gamma^{(j)}$.  Subtracting the second of these equalities  from the first, we have that  $\widetilde{d}_{\gamma,\tau}^{(j)}  = d^{(j)}_\tau + \delta^{(j)}_\gamma  \cdot \tau +\rho^{(j)}_\gamma  $.
But $d_\tau^{(1)}  = d_\tau^{(2)} $ and $ \widetilde{d}_{\gamma,\tau}^{(1)}  =  \widetilde{d}_{\gamma,\tau}^{(2)},$ which implies that
$ \delta^{(1)}_\gamma  \cdot \tau +\rho^{(1)}_\gamma =  \delta^{(2)}_\gamma  \cdot \tau +\rho^{(2)}_\gamma   $.  The second and third assertions now follow, using the fact that $\delta^{(j)}_\gamma$ and $\rho^{(j)}_\gamma$ are constants in $\mathbb R$, and $\tau \in \mathbb H$.
\end{proof}
 In order to utilize Lemma \ref{lem_transhat} to  determine the modular  transformation properties for the functions $V_{mn}$, we need to know for which $\gamma = \sm{a}{b}{c}{d} \in \textnormal{SL}_2(\mathbb Z)$ we have that $\widetilde{u}^{(mn)}_{\gamma,\tau} - u^{(mn)}_{\tau}\in \mathbb Z \tau +  \mathbb Z$, and $\widetilde{v}^{(mn)}_{\gamma,\tau} - v^{(mn)}_{\tau} \in \mathbb Z \tau +  \mathbb Z$.  We note the following lemma, which follows directly    by using the definition  of   $\widetilde{x}_{\gamma,\tau}.$ \begin{lemma}\label{lem_integers}
Let $x_{\tau} \in \mathbb C \setminus (\mathbb Z \tau +  \mathbb Z)$ be of the form
\[
x_{\tau} = \frac{\alpha \tau + \beta}{N},
\]
where $N\in\mathbb{N}$, and $1\leq \alpha, \beta \leq N-1$.  For fixed $\gamma = \sm{a}{b}{c}{d} \in \textnormal{SL}_2(\mathbb Z)$, we have that $\widetilde{x}_{\gamma, \tau} - x_{\tau} \in  \mathbb Z \tau +  \mathbb Z$ if and only if the following congruences hold
\begin{align*}
\alpha a + \beta c &\equiv \alpha \pmod{N}\\
\alpha b + \beta d &\equiv \beta \pmod{N}.
\end{align*}
\end{lemma}
\hfill \qed

The following corollary follows directly from Lemma \ref{lem_integers}.
\begin{corollary}\label{cor_helper}
In the context of the above lemma, when $\alpha=0$, and $\beta$ is relatively prime to $N$, then $\widetilde{x}_{\gamma, \tau}  - x_{\tau}   \in  \mathbb Z \tau +  \mathbb Z$ if and only if  $\gamma\in\G_1(N)$.   Similarly, if $\beta=0$, and $\alpha$ is relatively prime to $N$, then $\widetilde{x}_{\gamma, \tau} - x_{\tau}   \in  \mathbb Z \tau +  \mathbb Z$ if and only if   $\gamma\in\G^1(N)$. 
\end{corollary}
\hfill \qed

In Table \ref{Atable}, we list the congruence subgroups $A_{mn}$ for each mock modular form $V_{mn}$ in Theorem 1.1. These are computed using Lemma \ref{lem_integers} and Corollary \ref{cor_helper}, and are used in the proof of Theorem \ref{thm_modvij} below.

\begin{table}[h] 
\caption{congruence subgroups $A_{mn}$ for each mock modular form $V_{mn}$}\label{Atable}
\[
\begin{array}{|c|c|c|c|c|c|c|}
\hline
n\backslash m&1&2&3&4,4',4''&5&6\\
\hline
1 &\G^1(4)\cap \G_0(2)&\G^1(4)&\G^1(6)\cap \G_0(2) &\G^1(12)& \G^1(6)\cap \G_0(2) & \G^1(6)\\
\hline
2 &\G^1(4)\cap \G_0(2)&\G^1(4)\cap\G_0(2)& \G^1(6)\cap \G_0(2)&\G^1(12)\cap \G_0(2) & \G^1(6)\cap \G_0(2)& \G^1(6)\cap \G_0(2) \\
\hline
3 &\G^1(12)\cap \G_0(2)&\G^1(12)& \G^1(6)\cap \G_0(2)&\G^1(12)&\G^1(3)\cap \G_0(2) & \G^1(6)\\
\hline
4 &\G^1(12)\cap \G_0(2)&\G^1(12)\cap \G_0(2)&\G^1(6)\cap \G_0(2) &\G^1(12)\cap \G_0(2)&--& \G^1(6)\cap \G_0(2) \\
\hline
5 &\G^1(4)\cap \G_0(2)&--& \G^1(12)\cap \G_0(2) &\G^1(12)&\G^1(12)\cap \G_0(2)& \G^1(12) \\
\hline
6 &--& \G^1(4)\cap \G_0(2)& \G^1(12)\cap \G_0(2) &\G^1(12)\cap \G_0(2)& \G^1(12)\cap \G_0(2) & \G^1(12)\cap \G_0(2) \\
\hline
7 &\G^1(12)\cap \G_0(2)&\G^1(12)& \G^1(6)\cap \G_0(2) &\G^1(12)& \G^1(6)\cap \G_0(2) &--\\
\hline
8 &\G^1(12)\cap \G_0(2)& \G^1(12)\cap \G_0(2)&--&\G^1(12)\cap \G_0(2) & \G^1(6)\cap \G_0(2) & \G^1(6)\cap \G_0(2) \\
\hline
\end{array}
\]
\end{table}

\begin{proof}[Proof of Theorem \ref{thm_modvij}]
 We first consider the functions $V_{mn}$ where $m\in T'\setminus\{4\}$.  After doing so, we will  address the more delicate case of $m=4$. We begin by considering (for $m\in T'\setminus\{4\}$) the defining parameters $u_\tau^{(mn)}$ and $v_\tau^{(mn)}$ from the Appendix, as well as their associated values  $\widetilde{u}^{(mn)}_{\gamma,\tau}$ and $\widetilde{v}^{(mn)}_{\gamma,\tau}$,   where $\gamma = \sm{a}{b}{c}{d} \in \textnormal{SL}_2(\mathbb Z)$.  For $\gamma \in A_{mn}$ as  defined  in Table \ref{Atable}, we may write
\begin{align*} \widetilde{u}^{(mn)}_{\gamma,\tau} &= u^{(mn)}_\tau+ k^{(mn)}_{\gamma} \cdot \tau  + \ell^{(mn)}_\gamma,
\\ \widetilde{v}^{(mn)}_{\gamma,\tau} &= v^{(mn)}_\tau+ r^{(mn)}_{\gamma} \cdot \tau  + s^{(mn)}_\gamma,\end{align*}
where $k^{(mn)}_{\gamma}, \ell^{(mn)}_{\gamma}, r^{(mn)}_{\gamma}, s^{(mn)}_{\gamma}\in\mathbb{Z}$.  For example, when $(m,n) = (2,2),$  we have that
$u^{(22)}_{\tau} := \frac{\tau}{4} - \frac12 = \frac{\tau-2}{4}$ and $v^{(22)}_\tau := \frac{\tau}{2} - \frac12 = \frac{\tau-1}{2}$.  By Lemma \ref{lem_integers}, we see that $\widetilde{u}^{(22)}_{\gamma,\tau} -u^{(22)}_{\tau}\in  \mathbb Z \tau +  \mathbb Z$ if and only if
\begin{align*}
a + 2c &\equiv 1 \pmod{4}\\
b + 2d &\equiv 2 \pmod{4},
\end{align*}
whereas $\widetilde{v}^{(22)}_{\gamma,\tau} -v^{(22)}_{\tau}\in  \mathbb Z \tau +  \mathbb Z$ if and only if
\begin{align*}
a + c &\equiv 1 \pmod{2}\\
b + d &\equiv 1 \pmod{2}.
\end{align*}
Recalling that $ad-bc=1$, a straightforward calculation shows that these congruences are simultaneously satisfied if and only if $\gamma\in\G^1(4)\cap\G_1(2)$, which is $A_{22}$ in Table \ref{Atable}.

Thus, for general $m\in T'\setminus\{4\}$ we may apply Lemma \ref{lem_transhat}, which reveals that
\begin{align*}
 \widehat{V}_{mn}(\gamma_{mn} \tau) = \psi(\gamma_{mn})^{-3} (-1)^{k^{(mn)}_{\gamma_{mn}} + \ell^{(mn)}_{\gamma_{mn}}+r^{(mn)}_{\gamma_{mn}}+ s^{(mn)}_{\gamma_{mn}}}(c_{mn}\tau+d_{mn})^{\frac12} \phi^{(m)}_{n,\gamma_{mn},\tau} \widehat{V}_{mn}(\tau),
\end{align*}  where for $\gamma = \sm{a}{b}{c}{d} \in \textnormal{SL}_2(\mathbb Z)$, the functions $\phi^{(m)}_{n,\gamma,\tau}$ are defined by \begin{align}    \label{def_fgamma}
\phi^{(m)}_{n,\gamma,\tau} & := e \left(t_m \gamma \tau \right)e\left(\frac{-c}{2(c\tau+d)}(\tilde{u}^{(mn)}_{\gamma,\tau} -\tilde{v}^{(mn)}_{\gamma,\tau})^2\right)  q^{\frac12(k^{(mn)}_\gamma-r^{(mn)}_\gamma)^2}e\left(( {u}^{(mn)}_\tau -  {v}^{(mn)}_\tau)(k^{(mn)}_\gamma-r^{(mn)}_\gamma)\right)q^{-t_m}.
\end{align}
Next,  we define the difference functions
\begin{align*}d_{\tau}^{(mn)} := u^{(mn)}_{\tau} - v^{(mn)}_{\tau},  \ \   \widetilde{d}^{(mn)}_{\gamma,\tau}  := \widetilde{u}_{\gamma,\tau}^{(mn)} - \widetilde{v}_{\gamma,\tau}^{(mn)}, \ \
  \delta^{(mn)}_{\gamma} : =   k_{\gamma}^{(mn)}  - r_{\gamma}^{(mn)}.\end{align*} By hypotheses,  we have that  $d_{\tau}^{(mn)} = D^{(m)}_\tau$, for some function $D^{(m)}_\tau$, which is independent of $n$.  Thus, by Lemma \ref{lem_invuv}, we have for any $n$ such that $(m,n)$ is an admissible pair  that $ \widetilde{d}^{(mn)}_{\gamma,\tau} = \widetilde{D}^{(m)}_{\gamma,\tau}$ and $\delta^{(mn)}_\gamma = \Delta^{(m)}_\gamma$, for some functions $\widetilde{D}^{(m)}_{\gamma,\tau}$ and $\Delta^{(m)}_\gamma$ which are independent of $n$.  For example, when $m=2$, we have that
  \begin{align*} D^{(2)}_\tau = -\frac{\tau}{4}, \ \ \widetilde{D}^{(2)}_{\gamma,\tau} = -\frac14(a\tau+b), \ \ \Delta^{(2)}_\gamma = \frac{1-a}{4}.
  \end{align*}
   Thus, the functions $\phi^{(m)}_{n,\gamma,\tau}$ defined in (\ref{def_fgamma}) are in fact independent of $n$;  that is,  $\phi^{(m)}_{n,\gamma,\tau} = \Phi^{(m)}_{\gamma,\tau}$, where
 \begin{align*}
 \Phi^{(m)}_{\gamma,\tau}  & := e \left(t_m\gamma\tau \right)e\left(\frac{-c}{2(c\tau+d)}(\widetilde{D}^{(m)}_{\gamma,\tau})^2\right)  q^{\frac12(\Delta_\gamma^{(m)})^2}e\left( D_\tau^{(m)} \Delta_\gamma^{(m)}\right)q^{-t_m}.
\end{align*}
After some simplification, using the fact that  $\det\gamma=1$ for any $\gamma \in \textnormal{SL}_2(\mathbb Z)$, we find that $\Phi^{(m)}_{\gamma,\tau} = \varepsilon_\gamma^{(m)}$, where  \begin{align*}
\varepsilon_\gamma^{(m)} &:= \begin{cases} e\left(a b t_m\right), & m=2,4',4'',6, \\
e\left(\frac{4-4a-a b + 4c  }{32}\right), & m=1, \\
e\left(\frac{6-6a-a b + 18 c   - 9 cd }{72 }\right), & m=3, \\
e\left(\frac{12-12 a - 4 a b + 18 c   - 9 c d }{72}\right), & m=5, \end{cases}
\end{align*} and in particular, is a root of unity, and thus also independent of $\tau$.
Thus, we have  shown for $m\in T'\setminus\{4\}$ that \begin{align}\label{eqn_vmnhattr}
\widehat{V}_{mn}(\gamma_{mn}\tau) &=  \psi(\gamma_{mn})^{-3} (-1)^{k^{{(mn)}}_{\gamma_{mn}}+\ell^{(mn)}_{\gamma_{mn}}+r^{(mn)}_{\gamma_{mn}}+s^{(mn)}_{\gamma_{mn}}}\varepsilon_{\gamma_{mn}}^{(m)}(c_{mn}\tau+d_{mn})^{\frac12}  \widehat{V}_{mn}(\tau),\end{align} as desired.

When $m=4$, some additional care is required, as the function $V_{4n}$ is formed by adding $V_{4'n}$ and $V_{4''n}$.  While the groups $A_{4'n}$ and $A_{4''n}$ are equal, a priori, it is not clear  for a matrix $\gamma_n = \sm{a}{b}{c}{d} \in A_{4'n}=A_{4''n} = A_{4n}$ that the two multipliers
$$
(-1)^{k^{{(4'n)}}_{\gamma_n}+\ell^{(4'n)}_{\gamma_{n}}+r^{(4'n)}_{\gamma_{n}}+s^{(4'n)}_{\gamma_{n}}}\varepsilon_{\gamma_{n}}^{(4')} \ \ \text{ and } \ \
(-1)^{k^{{(4''n)}}_{\gamma_n}+\ell^{(4''n)}_{\gamma_{n}}+r^{(4''n)}_{\gamma_{n}}+s^{(4''n)}_{\gamma_{n}}}\varepsilon_{\gamma_{n}}^{(4'')}
$$
are equal, which we desire in order to give a transformation property for the functions $\widehat{V}_{4n}$ by adding the transformations in (\ref{eqn_vmnhattr}) when $m=4'$ and $m=4''$.   By parity considerations, a direct calculation reveals that indeed, $(-1)^{k^{{(4'n)}}_{\gamma_n}+\ell^{(4'n)}_{\gamma_{n}}+r^{(4'n)}_{\gamma_{n}}+s^{(4'n)}_{\gamma_{n}}} = (-1)^{k^{{(4''n)}}_{\gamma_n}+\ell^{(4''n)}_{\gamma_{n}}+r^{(4''n)}_{\gamma_{n}}+s^{(4''n)}_{\gamma_{n}}}$.  For the remaining roots of unity we use fact that $a=1+12a'$ and $b=12b'$  for some integers $a'$ and $b'$, and hence,
\begin{align*} \varepsilon_{\gamma_{n}}^{(4'')} &= \zeta_{288}^{-(1+12a')12b'} = \zeta_{24}^{-b'}(-1)^{-a'b'} =  \zeta_{24}^{-25b'}(-1)^{-25a'b'} = \zeta_{288}^{-25(1+12a')12b'} =
\varepsilon_{\gamma_n}^{(4')}. \end{align*}   Thus, when $m=4$, (\ref{eqn_vmnhattr}) holds as well, with
the multiplier \begin{align*}(-1)^{k^{{(4n)}}_{\gamma_{4n}}+\ell^{(4n)}_{\gamma_{4n}}+r^{(4n)}_{\gamma_{4n}}+s^{(4n)}_{\gamma_{4n}}}\varepsilon_{\gamma_{4n}}^{(4)} := (-1)^{k^{{(4'n)}}_{\gamma_{4n}}+\ell^{(4'n)}_{\gamma_{4n}}+r^{(4'n)}_{\gamma_{4n}}+s^{(4'n)}_{\gamma_{4n}}}\varepsilon_{\gamma_{4n}}^{(4')}
= (-1)^{k^{{(4''n)}}_{\gamma_{4n}}+\ell^{(4''n)}_{\gamma_{4n}}+r^{(4''n)}_{\gamma_{4n}}+s^{(4''n)}_{\gamma_{4n}}}\varepsilon_{\gamma_{4n}}^{(4'')}. \end{align*}
To show that the functions $\widehat{V}_{mn}$ are harmonic Maass forms, we must additionally show that they are annihilated by the operator $\Delta_{\frac12}$.  As summarized in the Appendix, we have for any admissible pair $(m,n)$ that the function $\widehat{V}_{mn}(\tau)$ may be expressed, up to multiplication by an easily determined constant $\alpha_{mn}$, as follows:  \begin{align} \label{table_VM}\begin{array}{lllllllll}  \widehat{V}_{1n}(\tau)  &=&  \alpha_{1n} \widehat{M}_{-\frac14,-\frac12}(\tau),&\widehat{V}_{4'n}(\tau) &=&  \alpha_{4'n} \widehat{M}_{-\frac{5}{12},0}(\tau), &\widehat{V}_{5n}(\tau) &=&  \alpha_{5n} \widehat{M}_{-\frac13,-\frac12}(\tau), \\
 \widehat{V}_{2n}(\tau) &= & \alpha_{2n} \widehat{M}_{-\frac14,0}(\tau),& \widehat{V}_{4''n}(\tau) &=&  \alpha_{4''n} \widehat{M}_{-\frac{1}{12},0}(\tau),  &\widehat{V}_{6n}(\tau) &= & \alpha_{6n} \widehat{M}_{-\frac16,0}(\tau), \\
  \widehat{V}_{3n}(\tau) &= & \alpha_{3n} \widehat{M}_{-\frac16,-\frac12}(\tau), & \widehat{V}_{4n}(\tau) &=& \alpha_{4n}\left(\widehat{M}_{-\frac{5}{12},0}(\tau) + \widehat{M}_{-\frac{1}{12},0}(\tau)\right),  \end{array}
\end{align}
where the functions $\widehat{M}_{a,b}(\tau)$ are defined in (\ref{def_Mhat}).
We then apply Proposition \ref{prop_p1shadows} to see that the functions $\widehat{V}_{mn}$ are annihilated by the operator $\Delta_{\frac12}$.   That these forms satisfy adequate growth conditions follows from their definitions.  Clearly, the functions $V_{mn}$ are the holomorphic parts of the forms $\widehat{V}_{mn}$, hence, are mock modular.

Finally, we prove that for $m\in T$, the functions $V_{mn}$ have, up to a constant multiple, shadows given by the weight $3/2$ eta-theta functions $E_m\left(\frac{2\tau}{c_m^2}\right)$.  To show this, we use (\ref{table_VM}), Proposition \ref{prop_p1shadows}, and Lemma \ref{E-g}.   In the case of $m=1$, combining (\ref{table_VM}) with Proposition
\ref{prop_p1shadows} part (i) shows that up to a constant, the mock modular forms $V_{1n}$ have shadows given by
$g^c_{\frac14,0}(\tau).$  It is not difficult to show by definition that $g^c_{\frac14,0}(\tau) = g_{\frac14,0}(\tau).$  We previously established in Lemma \ref{E-g}  that $E_1(\tau/32) = 4 g_{\frac14,0}(\tau),$ hence, we have proved that the functions $V_{1n}(\tau)$ have shadows given by a (computable)  constant multiple of the eta-theta function $E_1(\tau/32)$, as claimed.  The analogous results for the functions $V_{mn}$ for the other values of $m$ follow by a similar argument.\end{proof}

\section{Quantum sets}\label{sec_quantumsets}

In order to establish quantum modularity of the functions $V_{mn}$,  we must first determine viable sets of rationals.  We call a subset $S\subseteq\mathbb{Q}$ a {\it quantum set} for a function $F$ with respect to the group $G\subseteq \rm{SL}_2(\mathbb{Z})$ if both $F(x)$ and $F(Mx)$ exist (are non-singular) for all $x\in S$ and $M\in G$.

\subsection{Utilizing Theorem \ref{kangthm}}

By examining our catalogue of $V_{mn}$ in the Appendix, we see that precisely when $n=1$ we have a $\mu$-function in the form given in Theorem \ref{kangthm} of the third author. Using Theorem \ref{kangthm} in these cases, and the notation from the Appendix, we directly  obtain the following lemma.

\begin{lemma}\label{Vm1}
For $m\in T'\backslash \{4\}$, we have that
\[
V_{m1}(\tau) = i w_mq^{t_m+\frac18} \cdot g_2\left(e\left( \frac{u_{\tau}^{(m1)}}{2} \right) ;q^{\frac12} \right) -w_m q^{t_m + \frac18} e\left( -\frac{u_{\tau}^{(m1)} }{2}\right)\cdot \frac{\eta(\tau)^4}{\eta(\frac{\tau}{2})^2 \vartheta(u_{\tau}^{(m1)};\tau)}.
\]
\end{lemma}

Furthermore, when $n\neq 1$ we are still able to utilize Theorem \ref{kangthm}.  Let $m\in T'\backslash \{4\}$, and $n$ admissible.  Then by Lemma \ref{lem_mu} (4),  and using the fact that $\vartheta'(0;\tau) = -2\pi \eta(\tau)^3$, we see that
\begin{equation}\label{def_fmndef}
  V_{mn}(\tau) - V_{m1}(\tau) =  iw_m q^{t_m}\frac{\eta^3(\tau) \vartheta\left(\frac{\tau}{2} + u_\tau^{(mn)};\tau\right)\vartheta(u_\tau^{(mn)} - u_\tau^{(m1)};\tau)}{\vartheta(u_\tau^{(m1)};\tau) \vartheta\left(\frac{\tau}{2};\tau\right) \vartheta(u_\tau^{(mn)};\tau) \vartheta(v_\tau^{(mn)};\tau)}  =:  \mathcal F_{mn}(\tau).   \end{equation}
We will  explicitly  show in Lemma \ref{lem_modfmn} that these functions $\mathcal F_{mn}(\tau)$  transform like  weakly holomorphic modular forms of weight $1/2$.
Since we can write 
\begin{equation}\label{eq_KangExtm}
V_{mn}(\tau) = V_{m1}(\tau) + \mathcal F_{mn}(\tau)
\end{equation}
for $m\in T'\backslash \{4\}$, and for $m=4$, 
\begin{equation}\label{eq_KangExt4}
V_{4n}(\tau) = V_{41}(\tau) + \mathcal F_{4'n}(\tau) + \mathcal F_{4''n}(\tau),
\end{equation}
once we establish quantum sets for the $V_{m1}$, which we will do in Section \ref{sec_quantset1}, we can use \eqref{eq_KangExtm} and \eqref{eq_KangExt4} to find subsets that are quantum sets for the more general $V_{mn}$, for each $m\in T$. We do so in Section \ref{sec_qsetvmn}. 

\subsection{Determining quantum sets for $V_{m1}$}\label{sec_quantset1}
Observe from Lemma \ref{Vm1} that for $m\in T'\backslash\{4\}$, $V_{m1}(\tau)$ is a sum of two of terms, each with a coefficient that is a constant times a power of $q$.  The first term is of the form
\begin{equation}\label{eqn:g_2part}
g_2(a_mq^\frac{1}{2b_m};q^\frac12)  =\sum_{n=0}^\infty \frac{(-q^\frac12;q^\frac12)_n q^{n(n+1)/4}}{(a_mq^\frac{1}{2b_m};q^\frac12)_{n+1}(a_m^{-1}q^\frac{b_m-1}{2b_m};q^\frac12)_{n+1}},
\end{equation}
where $a_m = 1$ when $m$ is even, $a_m=i$ when $m$ is odd, and $(b_1,b_2, b_3, b_{4'}, b_{4''}, b_5, b_6) =(4,4,3,12,12/5, 6,3)$.  
The second term is
\[
\frac{\eta(\tau)^4}{\eta(\frac{\tau}{2})^2 \vartheta(u_{\tau}^{(m1)} ; \tau)},
\]
which by  \eqref{eq:theta}  we see is equal to   $ie(u_{\tau}^{(m1)})$  multiplied by  the infinite product
\begin{align}\label{eqn:eta_part}
f_m(\tau) := \frac{(q;q)_\infty( -  q^\frac12;q^\frac12)_\infty^2}{(a_m^2q^\frac{1}{b_m};q)_\infty(a_m^{-2}q^{\frac{b_m-1}{b_m}};q)_\infty}.
\end{align}
Note that for any $\tau\in\mathbb{Q}$, constant and $q$-power multiples of these terms will not affect whether $V_{m1}(\tau)$ and $V_{m1}(M\tau)$ exist.  Thus, we may determine quantum sets for each $V_{m1}(\tau)$ by examining the sum and product appearing in equations \eqref{eqn:g_2part} and \eqref{eqn:eta_part}.
We seek rational numbers $h/k \in \mathbb Q$ such that for sufficiently large $n$, 
\begin{align}\label{eqn_terminate}
0 = \left(-e\left(\frac{h}{2k}\right);e\left(\frac{h}{2k}\right)\right)_n = \prod_{j=1}^n \left(1 + e\left(\frac{jh}{2k}\right)\right),
\end{align}
and hence, the infinite sum defining the function $g_2$ in (\ref{eqn:g_2part}) terminates, and can be explicitly evaluated.  
The identity in (\ref{eqn_terminate})  holds if and only if $jh/k$ is an odd integer for some $1\leq j\leq n$.  This can never happen when $h$ is even, and when $h$ is odd, then $j=k$ causes the series to terminate at $n=k$.  Thus, the largest possible set  for which (\ref{eqn_terminate}) can hold is
\begin{equation}\label{def_S}
S:=\left\{h/k\in\mathbb{Q} \, \middle| \, h\in\mathbb{Z}, k\in\mathbb{N}, \gcd(h,k)=1,  h\equiv 1 \!\!\!\! \pmod 2 \right\}.
\end{equation}
We also set 
  \begin{align}\label{def_moreS}
S' &:= \left\{ h/k \in S \mid h\equiv \pm 1 \!\!\!\!  \pmod 6\right\} \\
 S_{ev} & := \left\{  h/k \in S \mid k\equiv 0 \!\!\!\!  \pmod 2 \right\} \nonumber \\
S_{od} & := \left\{  h/k \in S \mid k\equiv 1 \!\!\!\!  \pmod 2 \right\}, \nonumber 
\end{align}  
and define the subsets $S_{m1}\subseteq S$ by  
\begin{align*}
 S_{11}, \, S_{21}, \,  S_{41}   &:=S\\
 S_{31}, S_{61} & := S'  \\
 S_{51} & :=   S'\cup S_{ev}.  
\end{align*}
We will prove the following theorem.
\begin{theorem}\label{thm_QuantumSet}
For $m\in T$, the set $S_{m1}$ is a quantum set for $V_{m1}$ with respect to the group $G_{m1}$, where
\begin{align*}
& G_{11} :=\left<\left(\begin{smallmatrix} 1 & 0\\ 2 & 1 \end{smallmatrix}\right), \left(\begin{smallmatrix} 1 & 4\\ 0 & 1 \end{smallmatrix}\right) \right> \subset \Gamma_0(2) \cap \Gamma^0(4),\\
& G_{21} :=\left<\left(\begin{smallmatrix} 1 & 0\\ 1 & 1 \end{smallmatrix}\right), \left(\begin{smallmatrix} 1 & 4\\ 0 & 1 \end{smallmatrix}\right) \right> \subset \Gamma^0(4),\\
& G_{31}, G_{51} :=\left<\left(\begin{smallmatrix} 1 & 0\\ 2 & 1 \end{smallmatrix}\right), \left(\begin{smallmatrix} 1 & 6\\ 0 & 1 \end{smallmatrix}\right) \right> \subset \Gamma_0(2) \cap \Gamma^0(6),\\
& G_{41} :=\left<\left(\begin{smallmatrix} 1 & 0\\ 1 & 1 \end{smallmatrix}\right), \left(\begin{smallmatrix} 1 & 12\\ 0 & 1 \end{smallmatrix}\right) \right> \subset \Gamma^0(12),\\
& G_{61} :=\left<\left(\begin{smallmatrix} 1 & 0\\ 1 & 1 \end{smallmatrix}\right), \left(\begin{smallmatrix} 1 & 6\\ 0 & 1 \end{smallmatrix}\right) \right> \subset \Gamma^0(6).\\
\end{align*}
\end{theorem}

Before proving Theorem \ref{thm_QuantumSet} we prove two lemmas which analyze the behavior of $V_{m1}$ on $S_{m1}$.

\begin{lemma}\label{lem_g_2}
For each $m\in T'\backslash\{4\}$ we have that $g_2(a_mq^\frac{1}{2b_m};q^\frac12)$ is well-defined for $\tau\in S_{m1}$.
\end{lemma}

\begin{proof}
We have seen above that for any $\tau\in S$, the series in \eqref{eqn:g_2part} terminates at $n=k$.  We further require that $(a_m e(\frac{h}{2b_m k});e(\frac{h}{2k}))_{n+1}$, and $(a_m^{-1} e(\frac{(b_m-1)h}{2b_m k});e(\frac{h}{2k}))_{n+1}$ do not vanish before the termination of the series.  First, we note that
\begin{align*}
\left(a_m e\left(\frac{h}{2b_m k}\right);e\left(\frac{h}{2k}\right)\right)_{n+1} & = \prod_{j=0}^n \left(1-a_m e\left(\frac{(b_mj+1)h}{2b_mk} \right) \right) \\
\left(a_m^{-1} e\left(\frac{(b_m-1)h}{2b_m k}\right);e\left(\frac{h}{2k}\right)\right)_{n+1} & = \prod_{j=0}^n \left(1-a_m^{-1} e\left(\frac{(b_m(j+1)-1)h}{2b_mk} \right) \right).
\end{align*}

When $m$ is even, we have $a_m=a_m^{-1}=1$, and when $m$ is odd, we have $a_m = i$ and $a_m^{-1}=-i$.  Thus for $m$ even we need to avoid the existence of an $r\in\mathbb{Z}$ and $0\leq j \leq n$ such at least one of the following hold,
\begin{align}
h(b_mj+1) &= 2b_m kr, \label{even1} \\
h(b_m(j+1) -1) &= 2b_m kr. \label{even2}
\end{align}

This can never occur for the cases $m=2,4'$ because $b_2=4$ and $b_{4'}=12$ are even while $h$ is odd.  When $m=4''$, multiplying the equations through by $5$ gives a similar contradiction since $5b_{4''}=12$ is even while $5$, $h$ are odd.  Thus when $m=2,4',4''$, $S$ is the largest set  of rationals over which  the sum defining $g_2(a_mq^\frac{1}{2b_m};q^\frac12)$ terminates. 

For $m=6$, we have $b_6=3$, so we see that one of $h(3j+1) = 6kr$ or $h(3j+2) = 6kr$ can occur when $h\equiv 3\pmod{6}$.  This is because we must have that $k\equiv \pm 1\pmod{3}$, so if $k\equiv 1\pmod 3$, let $j=\frac{2k-1}{3}$, then we have that $0<j<k$ is an integer and so is $r = \frac{h(3j+1)}{6k}$.  Similarly, if $k\equiv 2\pmod 3$, let $j=\frac{2k-2}{3}$, then we have that $0<j<k$ is an integer and so is $r = \frac{h(3(j+1)-1)}{6k}$.  However, when $h\equiv \pm1 \pmod{6}$ we see that neither \eqref{even1} nor \eqref{even2} can be satisfied by reducing the equalities modulo 3.  Thus when $m=6$, $S_{61}$ is the largest  set of rationals over which  the sum defining  $g_2(a_mq^\frac{1}{2b_m};q^\frac12)$   terminates.  

Similarly, when $m$ is odd we   wish   to avoid the existence of an $r\in\mathbb{Z}$ and $0\leq j \leq n$ such that at least one of the following hold,
\begin{align}
2h(b_mj+1) &= b_m k(4r-1), \label{odd1} \\
2h(b_m(j+1) -1) &= b_m k(4r+1). \label{odd2}
\end{align}

This can never occur for the case $m=1$ because here $b_1=4$ while $h$ is odd, meaning the left hand side of neither equation is divisible by $4$.  Thus when $m=1$, $S$ is the largest  set of rationals over which  the sum defining  $g_2(a_mq^\frac{1}{2b_m};q^\frac12)$  terminates.  

For $m=3$, we have $b_3=3$ so equations \eqref{odd1}, \eqref{odd2} become
\begin{align*}
2h(3j+1) &= 3k(4r-1), \\
2h(3j +2) &= 3k(4r+1).
\end{align*}
When $h\equiv \pm1 \pmod{6}$, or if $h\equiv 3 \pmod{6}$ and $k$ is odd, we see that neither \eqref{odd1} nor \eqref{odd2} can be satisfied.  However if $h\equiv 3 \pmod{6}$ and $k$ is even, then one of $2h(3j+1) = 3k(4r-1)$ or $2h(3j -2) = 3k(4r+1)$ can occur for some $0\leq j<k$.  This is because then $k\equiv \pm 2\pmod{6}$, so we may consider the following four possible cases.
\begin{enumerate}
\item Let $h\equiv 3 \pmod{12}$ and $k\equiv 2 \pmod 6$.  Then $j=\frac{5k-4}{6}\in\mathbb{N}$, and $r=\frac{1}{4}(\frac{2h(3j+2)}{3k} -1)\in\mathbb{Z}$.
\item Let $h\equiv 3 \pmod{12}$ and $k\equiv 4 \pmod 6$.  Then $j=\frac{k-4}{6}\in\mathbb{N}$, and $r=\frac{1}{4}(\frac{2h(3j+2)}{3k} -1)\in\mathbb{Z}$.
\item Let $h\equiv 9 \pmod{12}$ and $k\equiv 2 \pmod 6$.  Then $j=\frac{k-2}{6}\in\mathbb{N}$, and $r=\frac{1}{4}(\frac{2h(3j+1)}{3k} +1)\in\mathbb{Z}$.
\item Let $h\equiv 9 \pmod{12}$ and $k\equiv 4 \pmod 6$.  Then $j=\frac{5k-2}{6}\in\mathbb{N}$, and $r=\frac{1}{4}(\frac{2h(3j+1)}{3k} +1)\in\mathbb{Z}$.
\end{enumerate}
In each of these cases observe that $0\leq j <k$.  Thus we see that when $m=3$, actually  $S'\cup S_{od}$  
is the largest  set of rationals over which   the sum defining  $g_2(a_mq^\frac{1}{2b_m};q^\frac12)$   terminates. However, we will see in the next lemma that we must eventually restrict to $S_{31}$.

For $m=5$ we have $b_5=6$ so equations \eqref{odd1}, \eqref{odd2} become
\begin{align*}
h(6j+1) &= 3k(4r-1), \\
h(6j+5) &= 3k(4r+1).
\end{align*}
When $h\equiv \pm1 \pmod{6}$, or if $h\equiv 3 \pmod{6}$ and $k$ is even, we see that neither \eqref{odd1} nor \eqref{odd2} can be satisfied.  However if $h\equiv 3 \pmod{6}$ and $k$ is odd, then one of $h(6j+1) = 3k(4r-1)$ or $h(6j+5) = 3k(4r+1)$ can occur for some $0\leq j<k$.  This is because then $k\equiv \pm 1\pmod{6}$, so we may consider the following four possible cases.
\begin{enumerate}
\item Let $h\equiv 3 \pmod{12}$ and $k\equiv 1 \pmod 6$.  Then $j=\frac{5k-5}{6}\in\mathbb{N}$, and $r=\frac{1}{4}(\frac{h(6j+5)}{3k} -1)\in\mathbb{Z}$.
\item Let $h\equiv 3 \pmod{12}$ and $k\equiv 5 \pmod 6$.  Then $j=\frac{k-5}{6}\in\mathbb{N}$, and $r=\frac{1}{4}(\frac{h(6j+5)}{3k} -1)\in\mathbb{Z}$.
\item Let $h\equiv 9 \pmod{12}$ and $k\equiv 1 \pmod 6$.  Then $j=\frac{k-1}{6}\in\mathbb{N}$, and $r=\frac{1}{4}(\frac{h(6j+1)}{3k} +1)\in\mathbb{Z}$.
\item Let $h\equiv 9 \pmod{12}$ and $k\equiv 5 \pmod 6$.  Then $j=\frac{5k-1}{6}\in\mathbb{N}$, and $r=\frac{1}{4}(\frac{h(6j+1)}{3k} +1)\in\mathbb{Z}$.
\end{enumerate}
In each of these cases observe that $0\leq j <k$.  Thus we see that when $m=5$, $S_{51}$ is the largest  set of rationals over which  the sum defining $g_2(a_mq^\frac{1}{2b_m};q^\frac12)$ terminates.
\end{proof}

We next analyze the second term  from Lemma \ref{Vm1},  $f_m(\tau)$, when $\tau\in S_{m1}$.  The following result is used to prove the transformation formulas in the next section.

\begin{lemma}\label{lem_eta}
For each $m\in T'\backslash \{4\}$,  $f_m(\tau)$ vanishes for each $\tau\in S_{m1}$.
\end{lemma}

\begin{proof}
We observe from \eqref{eqn:eta_part} that the product  $(-q^{\frac12};q^{\frac12})_\infty   = \prod_{n\geq 1} (1+q^{\frac{n}{2}})$    appears in the numerator of $f_m(\tau)$.  Thus, as in our analysis of the $g_2$ term, we see that for $\tau=\frac{h}{k}\in S$  the $n=k$ term of this product will be $0$.  Similarly, the $n=k$ term of $(q;q)_\infty$ will also vanish.  Thus to show that $f_m(\tau)=0$ for $\tau\in S_{m1}$ it remains to show that the  products in the denominators of $f_m(\tau)$ are finite and nonzero on $S_{m1}$ for   terms  $1\leq n \leq k$  (when expressed as products indexed by $n$).  We see that the terms appearing in $(a_m^2q^\frac{1}{b_m};q)_\infty$ and $(a_m^{-2}q^{\frac{b_m-1}{b_m}};q)_\infty$, are the squares of terms appearing in the denominators of $g_2(a_mq^\frac{1}{2b_m};q^\frac12)$.  We analyze them similarly as in Lemma \ref{lem_g_2}.  For $\tau=\frac{h}{k}$, we have
\begin{align*}
\left((-1)^m e\left(\frac{h}{b_m k}\right);e\left(\frac{h}{k}\right)\right)_\infty & = \prod_{n\geq 0} \left(1-(-1)^m e\left(\frac{(b_mn+1)h}{b_mk} \right) \right) \\
\left((-1)^m e\left(\frac{(b_m-1)h}{b_m k}\right);e\left(\frac{h}{k}\right)\right)_\infty & = \prod_{n\geq 0} \left(1-(-1)^m e\left(\frac{(b_m(n+1)-1)h}{b_mk} \right) \right).
\end{align*}
Thus for $m$ even we wish   to avoid the existence of an $r\in\mathbb{Z}$ and $0\leq n \leq k$ such at least one of the following hold,
\begin{align*}
h(b_mn+1) &= b_m kr,  \\
h(b_m(n+1) -1) &= b_m kr.
\end{align*}
 For $m=2,4'$ we have that $b_m$ is even and $h$ is odd so this cannot occur.  When $m=4''$, multiplying the equations through by $5$ gives a similar contradiction since $5b_{4''}=12$ is even while $5$, $h$ are odd.  When $m=6$, we have $b_6=3$.  But for $\frac{h}{k}\in S_{61}$, we have $h\equiv \pm 1 \pmod{6}$ and so this can never occur.

When $m$ is odd, we must show there is no $r \in \Z$ such that
\begin{align}
h(b_mn+1) &= b_m k(2r+1), \label{etaodd1} \\
h(b_m(n+1) -1) &= b_m k(2r+1) \label{etaodd2}.
\end{align}
When $m=1,5$ we have that $b_m$ is even and $h$ is odd so this cannot occur. When $m=3$, we have $b_3=3$.  But for $\frac{h}{k}\in S_{31}$, we have that $h\equiv \pm 1 \pmod{6}$, so this can never occur due to different residues modulo $3$.  Notice it is here that we must restrict to $S_{31}$ from $S_{31}'$.  If $h\equiv 3\pmod{6}$ and $k$ is odd, then either $k\equiv 1 \pmod 3$, in which case we can let $n=\frac{k-1}{3}$ and $r=\frac{h-3}{6}$ in the first equation, or $k\equiv 2 \pmod 3$, in which case we can let $n=\frac{k-2}{3}$ and $r=\frac{h-3}{6}$ in the second equation.  Both instances result in a zero in the denominator before termination.
\end{proof}
 
\begin{remark}
Lemmas \ref{lem_g_2} and \ref{lem_eta} imply that for each $m\in T$, $S_{m1}$ is our  largest  possible quantum set for $V_{m1}$.
\end{remark} 

We now prove Theorem \ref{thm_QuantumSet}.

\begin{proof}(Proof of Theorem \ref{thm_QuantumSet})
 Let $m\in T$.   By Lemmas \ref{lem_g_2} and \ref{lem_eta}, we see that each $V_{m1}$ is well-defined for $\tau\in S_{m1}$, but it remains to  be seen that  $V_{m1}$ is well-defined for each $M\tau$, where $M\in G_{m1}$.  We conclude by proving that each set $S_{m1}$ is closed under transformations by the matrices in $G_{m1}$.  Observe that each $G_{m1}$ has two generators, one of the form $\left( \begin{smallmatrix} 1 & 0 \\ A & 1\end{smallmatrix} \right)$ and the other of the form $\left( \begin{smallmatrix} 1 & B \\ 0 & 1\end{smallmatrix} \right)$ for positive integers $A_m,B_m$.
For $h/k \in S$ we have
\[
T_{m,1}(h/k):=\begin{pmatrix} 1 & 0\\ A_m & 1\end{pmatrix}\frac{h}{k} = \frac{h}{k + A_mh}, \qquad
T_{m,2}(h/k):=\begin{pmatrix} 1 & B_m\\ 0 & 1\end{pmatrix}\frac{h}{k}  = \frac{h + B_mk}{k}.
\]
Since $\gcd(h,k)=1$, we have $\gcd(h, A_m k + h) = \gcd(h + B_mk,h) = 1$.  Moreover, we note that $B_m$ is even for each $m$, so when $h$ is odd we have that $h+B_mk$ is odd, and thus $T_m(h/k), T_m'(h/k) \in S$ for all $\tau\in S$.  Thus for $m=1, 2, 4', 4'', 4$ we have that $T_m(h/k), T_m'(h/k) \in S_{m1}$ for all $\tau\in S_{m1}$.  When $m=3,5,6$ we have $B_m=6$ so that $h+B_mk\equiv h \pmod{6}$.  Thus for $m=3,5,6$,  we see that $T_m'(h/k) \in S_{m1}$ for all $\tau\in S_{m1}$.  To see also that $T_m(h/k) \in S_{m1}$ for all $\tau\in S_{m1}$, we only need to observe that in the case $m=5$, when $k$ is even, then $k+2h$ is also even.

Now we need to also consider the inverses
\[
T_{m,1}^{-1}(h/k):=\begin{pmatrix} 1 & 0\\ -A_m & 1\end{pmatrix}\frac{h}{k} = \frac{h}{k - A_mh}, \qquad
T_{m,2}^{-1}(h/k):=\begin{pmatrix} 1 & -B_m\\ 0 & 1\end{pmatrix}\frac{h}{k}  = \frac{h - B_mk}{k}.
\]
When $k - A_mh$ is positive, the same arguments as above go through.  When $k - A_mh$ is negative, we observe that
\[
T_m^{-1}(h/k):=\begin{pmatrix} 1 & 0\\ -A_m & 1\end{pmatrix}\frac{h}{k} = \frac{-h}{A_mh-k},
\]
has a positive denominator, and so again we use the arguments above.
\end{proof}

\subsection{Determining  quantum sets for general $V_{mn}$}\label{sec_qsetvmn}

In Section \ref{sec_quantset1}, we  determined the quantum sets $S_{m1}$ for the function $V_{m1}$.  In this section, we will use \eqref{eq_KangExtm} and \eqref{eq_KangExt4} to determine the more general quantum sets $S_{mn}$ for the functions $V_{mn}$ with $n\neq 1$.  Observe that our previous discussion shows that we  must   require   $S_{mn}\subseteq S_{m1}$ for each $m \in T$. We define the sets $S_{mn}$ for any $m\in T$ and admissible $n$ below;  for completeness, we also include the sets $S_{m1}$ previously determined.  In Lemma \ref{lem_zerofmn}, we establish that these sets are indeed appropriate, by showing that the auxiliary functions $\mathcal F_{mn}$ appearing in  \eqref{eq_KangExtm} and \eqref{eq_KangExt4} vanish at any rational point in $S_{mn}$. 

 We define the 43 subsets  $S_{mn}\subseteq S_{m1}$ by
\begin{align*}
S_{11} , S_{17} , S_{21} , S_{27} , S_{41} , S_{45} , S_{47}  &:= S \\
S_{12} , S_{18} , S_{22} , S_{28} , S_{42} , S_{52}  & := S_{ev} \\
S_{13} , S_{23} , S_{31} ,  S_{34} ,  S_{35} , S_{36} , S_{43} ,  S_{53} ,  S_{61} ,   S_{63} , S_{64} , S_{65} , S_{66}  & := S' \\
S_{14} , S_{15} ,  S_{24}  , S_{26} &:= S_{od} \\
S_{32} , S_{33}  , S_{62} &:= S'\cap S_{ev} \\
S_{37} , S_{68} &:= S'\cap S_{od} \\
S_{44} & :=  S' \cup S_{od}  \\
S_{46} , S_{48} , S_{51} , S_{55} , S_{56} , S_{57} ,  S_{58} &:= S' \cup S_{ev}.
\end{align*}

 \begin{lemma}\label{lem_zerofmn}
 For $m \in T\backslash\{4\}$ and $\frac{h}{k} \in S_{mn}$,  or for $m\in\{4',4''\}$ and $\frac{h}{k} \in S_{4n}$, we have that $$\mathcal F_{mn}\left(\frac{h}{k}\right) = 0.$$
\end{lemma}

\begin{proof}[Proof of Lemma \ref{lem_zerofmn}]
Note that by applying the triple product formula (\ref{eq:theta}) to each of the $\vartheta$-functions appearing in the definition of $\mathcal F_{mn}$ in (\ref{def_fmndef}), we can cancel the four copies of $(q;q)_\infty$ appearing in the denominator with four of the five copies appearing in the numerator (three of which arise from the function $\eta^3$).   Thus we may write $\mathcal F_{mn}(\tau)$ as a constant multiple of $q$  multiplied by 
 \begin{align}\label{eqn_fmnsimple} 
\frac{(q;q)_\infty }{(q^{\frac12};q)_\infty^2 } \cdot \frac{(e(\frac{\tau}{2} + u_{\tau}^{(mn)});q)_\infty (e(-\frac{\tau}{2} - u_{\tau}^{(mn)});q)_\infty}{ (e(u_{\tau}^{(m1)}); q)_\infty (e(-u_{\tau}^{(m1)})q ; q)_\infty  (e(u_{\tau}^{(mn)}); q)_\infty (e(-u_{\tau}^{(mn)})q ; q)_\infty (e(v_{\tau}^{(mn)}); q)_\infty (e(-v_{\tau}^{(mn)})q ; q)_\infty}.
 \end{align} 
 
Observe that for any $\tau = h/k\in S$, we have that $(q;q)_\infty$ vanishes at the $k$th term  when expanded,   and $(q^{\frac12};q)_\infty^2$ never vanishes.  Moreover, we have already demonstrated in the proof of Lemma \ref{lem_eta} that $(e(u_{\tau}^{(m1)}); q)_\infty (e(-u_{\tau}^{(m1)})q ; q)_\infty $ does not vanish for $\tau = h/k\in S_{m1}$,    as this term appears in the denominator of $f_{m}$. Thus, it suffices to show that when $\tau = h/k \in S_{mn}$ each of the products in
\begin{equation}\label{eqn_fmndenom} 
(e(u_{\tau}^{(mn)}); q)_\infty (e(-u_{\tau}^{(mn)})q ; q)_\infty (e(v_{\tau}^{(mn)}); q)_\infty (e(-v_{\tau}^{(mn)})q ; q)_\infty
\end{equation} 
is non-vanishing for terms $1\leq  s   \leq k$,   when expanded as products indexed by $s$.  Next, we observe that $v_{\tau}^{(mn)}$ depends only on $n$. In particular, 
$v_{\tau}^{(mn)} = \frac{\tau}{c_n}$ when $n$ is odd, and $v_{\tau}^{(mn)} = \frac{\tau}{c_n} - \frac12$ when $n$ is even, where $(c_1,c_2,c_3,c_4,c_5, c_6, c_7, c_8) = (2,2,3,3,4,4,6,6)$.  Thus, for admissible pairs $(m,n)$,
\[
(e(v_{\tau}^{(mn)}); q)_\infty (e(-v_{\tau}^{(mn)})q ; q)_\infty = 
\begin{cases}
(q^{\frac{1}{c_n}};q)_\infty (q^{\frac{c_n-1}{c_n}};q)_\infty & \mbox{ when $n$ odd} \\
(-q^{\frac{1}{c_n}};q)_\infty (-q^{\frac{c_n-1}{c_n}};q)_\infty & \mbox{ when $n$ even}.
\end{cases}
\] 

When $n$ is odd, we have that for $\tau = h/k$,
\[
(q^{\frac{1}{c_n}};q)_\infty (q^{\frac{c_n-1}{c_n}};q)_\infty= \prod_{j\geq 0} \left(1-e\left(\frac{(c_nj+1)h}{c_n k} \right) \right)  \left(1-e\left(\frac{(c_n(j+1)-1)h}{c_nk} \right) \right).
\]
Thus for $n$ odd we wish  to avoid the existence of an $r\in\mathbb{Z}$ and $0\leq j \leq k$ such that at least one of the following hold:
\begin{align*}
h(c_nj+1) &= c_n kr,  \\
h(c_n(j+1) -1) &= c_n kr.
\end{align*}
For $n=1,5,7$ we have that $c_n$ is even and $h$ is odd so this cannot occur.  When $n=3$, we have $c_3=3$.  But for $\frac{h}{k}\in S_{31}$, we have $h\equiv \pm 1 \pmod{6}$ and so we again have that this cannot occur.

When $n$ is even, we have that for $\tau = h/k$,
\[
(-q^{\frac{1}{c_n}};q)_\infty (-q^{\frac{c_n-1}{c_n}};q)_\infty= \prod_{j\geq 0} \left(1-e\left(\frac{(c_nj+1)h}{c_n k}  -\frac12 \right) \right)  \left(1-e\left(\frac{(c_n(j+1)-1)h}{c_nk} - \frac12 \right) \right),
\]
so in this case we need to avoid the existence of an $r\in\mathbb{Z}$ and $0\leq j \leq k$ such that at least one of the following hold,
\begin{align*}
2h(c_nj+1) &= c_n k(2r+1),  \\
2h(c_n(j+1) -1) &= c_n k(2r+1).
\end{align*}
When $n=6$ we have $c_n=4$ and since $h$ is odd so this cannot occur for any element in $S$.  When $n=2$, both equations reduce to the equation $h(2j+1)=k(2j+1)$.  In the definitions of the sets $S_{m2}$, we see that in each case $k$ is even, and so this equation can never be satisfied for an element of $S_{m2}$.  When $n=4$, we have the equations  $2h(3j+1) = 3k(2r+1)$,  and $2h(3j +2) = 3k(2r+1)$.  We see that these can not be satisfied when $h\not \equiv 0 \pmod3$, or when $h\equiv 3 \pmod{6}$ and $k$ odd.  Thus, for elements of $S_{m4}$ they cannot be satisfied.  Similarly, when $n=8$, we have the equations $2h(6j+1) = 6k(2r+1)$,  and $2h(6j +5) = 3k(6r+1)$, which also can't be satisfied when $h\not \equiv 0 \pmod3$.  In this case, they also can't be satisfied when $h\equiv 3 \pmod{6}$ and $k$ even.  The definitions of $S_{m8}$ shows that we are always in one of these cases.  

Thus, we have reduced the problem to showing that when $\tau = h/k \in S_{mn},$ the products
\begin{equation}\label{eqn_fmndenom} 
(e(u_{\tau}^{(mn)}); q)_\infty (e(-u_{\tau}^{(mn)})q ; q)_\infty 
\end{equation} 
are non-vanishing    in their first $k$ terms when expanded. Although at first glance it would seem that we have many cases to consider, in fact we have already done most of the work, we just need to compare each case to the defined set $S_{mn}$.  Comparing the values of $u_{\tau}^{(mn)}$ when $m>1$ to the values of $v_{\tau}^{(mn)}$ that we have already considered, and using that $e(\frac12)=e(-\frac12)$, we see that there are only about a dozen left to consider.  Moreover, the cases that are merely a negative multiple can be reduced fairly easily to the original case.  Thus the only $u_{\tau}^{(mn)}$ we will consider here are $u_{\tau}^{(13)}=\tau/12 +1/2$, $u_{\tau}^{(14)}=\tau/12$, $u_{\tau}^{(15)}=1/2$, and $u_{\tau}^{(4''2)}=5\tau/12 -1/2$.

For $u_{\tau}^{(13)}=\tau/12 +1/2$, \eqref{eqn_fmndenom} becomes
\[
(-q^{\frac{1}{12}};q)_\infty (-q^{\frac{11}{12}};q)_\infty= \prod_{j\geq 0} \left(1-e\left(\frac{(12j+1)h}{12 k}  -\frac12 \right) \right)  \left(1-e\left(\frac{(12(j+1)-1)h}{12k} - \frac12 \right) \right),
\]
and so we  wish   to avoid the existence of an $r\in\mathbb{Z}$ and $0\leq j \leq k$ such that at least one of the following hold,
\begin{align*}
h(12j+1) &= 6k(2r+1),  \\
h(12(j+1) -1) &= 6k(2r+1).
\end{align*}
But since $h$ is odd this can never occur.  

For $u_{\tau}^{(14)}=\tau/12$, \eqref{eqn_fmndenom} becomes
\[
(q^{\frac{1}{12}};q)_\infty (q^{\frac{11}{12}};q)_\infty= \prod_{j\geq 0} \left(1-e\left(\frac{(12j+1)h}{12 k} \right) \right)  \left(1-e\left(\frac{(12(j+1)-1)h}{12 k} \right) \right),
\]
and so we wish  to avoid the existence of an $r\in\mathbb{Z}$ and $0\leq j \leq k$ such that at least one of the following hold,
\begin{align*}
h(12j+1) &= 12 kr,  \\
h(12(j+1) -1) &= 12 kr,
\end{align*}
which again can never occur since $h$ is odd.

For $u_{\tau}^{(15)}=1/2$, \eqref{eqn_fmndenom} becomes
\[
(-1;q)_\infty (-q;q)_\infty= 2\prod_{j\geq 0} \left(1-e\left(\frac{jh}{k} -\frac12\right)  \right)^2,
\]
so we wish  to avoid the existence of an $r\in\mathbb{Z}$ and $0\leq j \leq k$ such that $2hj=k(2r+1)$ holds, which can't occur in $S_{15}$ since $k$ is odd. 

Lastly, when $u_{\tau}^{(4''2)}=5\tau/12 -1/2$, \eqref{eqn_fmndenom} becomes
\[
(-q^{\frac{5}{12}};q)_\infty (-q^{\frac{7}{12}};q)_\infty= \prod_{j\geq 0} \left(1-e\left(\frac{(12j+5)h}{12 k}  -\frac12 \right) \right)  \left(1-e\left(\frac{(12(j+1)-5)h}{12k} - \frac12 \right) \right),
\]
and so we  wish  to avoid the existence of an $r\in\mathbb{Z}$ and $0\leq j \leq k$ such that at least one of the following hold,
\begin{align*}
h(12j+5) &= 6k(2r+1),  \\
h(12(j+1) -5) &= 6k(2r+1),
\end{align*}
which can never occur since $h$ is odd.
\end{proof}

\section{Quantum modularity of the $V_{mn}$}\label{sec_quantum}

We now make more precise the notion of a quantum modular form. For $k\in\frac12\mathbb Z$, a \emph{quantum modular form of weight $k$ on the set $S$ for the group $G$} is a complex-valued function $f$ such that $S$ is a quantum set for $f$ with respect to the group $G\subseteq \textnormal{SL}_2(\mathbb Z)$. Further, for all $\gamma = \sm{a}{b}{c}{d} \in G$, and for all $x\in S$   ($x\neq -\frac{d}{c}$),   the functions $$ h_{f,\gamma}(x) := f(x) - \epsilon(\gamma)(cx+d)^{-k}f\left(\frac{ax+b}{cx+d}\right)$$ are suitably continuous or analytic in $\mathbb R$, as defined by Zagier in \cite{Zagier}.  In this paper, we will consider real analytic functions $h_{f,\gamma}$.     The $\epsilon(\gamma)$ are appropriate complex numbers, such as those that arise naturally in the theory of half-integer weight modular forms.

In this section, we prove Theorem \ref{thm_quantum} and Proposition \ref{prop_eichlere}, the first of which in particular establishes the quantum modularity of the functions $V_{mn}$.    We begin by defining for $m\in T$ the numbers
\begin{align*}\begin{array}{llll}
\ell_m:=\begin{cases}2, & m=1,3,5, \\ 1, & m=2,4,6,\end{cases}&
a_m := \begin{cases}8, & m=1,2, \\ 3, & m=3,6, \\ 24, & m=4, \\ 12, & m=5, \end{cases} &
b_m := \begin{cases} \frac{a_m}{2}, & m=1,2,4,5, \\ 2a_m, & m=3,6, \end{cases} &
c_m := \begin{cases} a_m, & m=1,2,4,5, \\ 2a_m, & m=3,6,\end{cases}
\end{array}\end{align*}
 and let $\ell_{4'}=\ell_{4''} :=1, b_{4'}=b_{4''}:=12,$ and $a_{4'}=a_{4''}=24$.
 We define the following groups  \begin{align*}\begin{array}{ll}
  G_{12}, G_{15}, G_{22}, G_{26} &:=\left<\left(\begin{smallmatrix} 1 & 0\\ 2 & 1 \end{smallmatrix}\right), \left(\begin{smallmatrix} 1 & 4\\ 0 & 1 \end{smallmatrix}\right) \right> \subset \Gamma_0(2) \cap \Gamma^0(4),\\ \ \\
  \begin{array}{l} G_{13}, G_{14}, G_{17}, G_{18}, G_{23}, G_{24}, G_{27}  \\ \ \ \ \  G_{28}, G_{35}, G_{36},  G_{4n}, G_{55}, G_{56}, G_{65}, G_{66}\end{array} & :=\left<\left(\begin{smallmatrix} 1 & 0\\ 2 & 1 \end{smallmatrix}\right), \left(\begin{smallmatrix} 1 & 12\\ 0 & 1 \end{smallmatrix}\right) \right> \subset \Gamma_0(2) \cap \Gamma^0(12),\\ \ \\
 \begin{array}{l} G_{32} , G_{33} , G_{34} , G_{37}, G_{52}, G_{53} \\ \ \ \ \ G_{57}, G_{58}, G_{62}, G_{63}, G_{64}, G_{68}\end{array} & :=\left<\left(\begin{smallmatrix} 1 & 0\\ 2 & 1 \end{smallmatrix}\right), \left(\begin{smallmatrix} 1 & 6\\ 0 & 1 \end{smallmatrix}\right) \right> \subset \Gamma_0(2) \cap \Gamma^0(6).\end{array}
\end{align*}
The sets $G_{m1}$ are as defined in Theorem \ref{thm_QuantumSet}, and the sets $G_{4n}$  above are defined for any admissible $n$ when $m$ equals $4$.   We  also define the constants \begin{align*} \kappa_{1n} &:=\!\begin{cases} 1, & n\in\{2,5\}, \\ 3, & n \in\{3,4,7,8\}, \end{cases}  \ \
 \kappa_{2n} :=\!\begin{cases} 1, & n\in\{2,6\}, \\ 3, & n \in\{3,4,7,8\},  \end{cases}   \ \
 \kappa_{3n} :=\!\begin{cases} 1, & n\in\{2,3,4,7\}, \\ 2, & n \in\{5,6\},  \end{cases} \\
 \kappa_{5n} &:= \!\begin{cases} 1, & n\in\{2,3,7,8\}, \\ 2, & n \in\{5,6\}, \end{cases} \ \
 \kappa_{6n} := \!\begin{cases} 1, & n\in\{2,3,4,8\}, \\ 2, & n \in\{5,6\}, \end{cases}
 \end{align*}  as well as $\kappa_{m1}=\kappa_{4n} = \kappa_{4'n}=\kappa_{4''n} := 1$ for any admissible pair $(m,1)$, $(4,n)$, $(4',n)$ or $(4'',n)$.
We recall that for $r\in\mathbb Z$, we let $M_r:=\sm{1}{0}{r}{1}$.

In Section \ref{sec_qs1}, we first sketch the general proof  of Theorem \ref{thm_quantum}   when $m\in T$ and $n=1$, and then provide details for the case when $(m,n)=(1,1)$.   After establishing the result for these pairs $(m,n)$, in Section \ref{sec_qmfvmnall}, we deduce the result for all remaining pairs $(m,n)$.  In Section \ref{sec_qeichlere}, we prove Proposition \ref{prop_eichlere}.

\subsection{ Proof of Theorem \ref{thm_quantum} for $(m,n) = (m,1)$}\label{sec_qs1}
\begin{proof}[General Proof of Theorem \ref{thm_quantum}   when    $m\in T, n=1$]  For $ r  \in \mathbb N$ we have $M_{ r}= S T^{-{r  }}S^{-1}$,
where $S=\sm{0}{-1}{1}{0}$ and $T=\sm{1}{1}{0}{1}$, and we define $\tau_{ r  } := T^{-{ r }}S^{-1}\tau = -1/\tau -  r  $. Using the fact that $M_{r  }\tau = S\tau_{r  }$, we find by straightforward but lengthy calculations using the expressions for $V_{m1}$ given in (\ref{eq_Vnotation}) (and the Appendix) combined with
 Lemma \ref{lem_mu}   that

\begin{align}\label{eqn_FEmij} V_{m1   }(M_{\ell_m}\tau) = \zeta_8^{2-\ell_m}(\ell_m\tau+1)^{\frac12} V_{m1   }(\tau) + \mathcal I_m(\tau) + \mathcal J_m(\tau). \end{align}
The functions $\mathcal I_m$ and $\mathcal J_m$ are defined by Mordell integrals $h(z;\tau)$, which we then simplify to Eichler integrals of weight $3/2$ unary theta functions $g_{a,b}$, using either   Theorem \ref{thm_z116}   (for $m=2,4,6$) or Lemma \ref{lem_Zlem} (for $m=1,3,5$).  We summarize these facts in the following table:
{\begin{align*}\begin{array}{lclcl}
\mathcal I_1(\tau) &:=&\displaystyle -\frac{\zeta_8}{2i}   e\left(\tfrac{1}{8\tau}\right) \sqrt{-i\tau_2} \ h\left( \tfrac{\tau_2}{2}+\tfrac14;\tau_2\right) &=& \displaystyle\frac{i}{2}  \sqrt{  2\tau+1 }\displaystyle  \int_{\frac12}^{0}  \frac{ g_{\frac14,0}\left(u\right)}{\sqrt{-i(u+\tau)}}du  + \frac{i}{2} \sqrt{ -i\tau_2}, \\
 \mathcal J_1(\tau) &:=& \displaystyle \frac{1}{2i} q^{-\frac{1}{32}}\sqrt{2\tau+1} \  h\left(\tfrac{\tau}{4}-\tfrac12;\tau\right) &=&  \displaystyle \frac{i}{2 }  \sqrt{2\tau+1}   \displaystyle \int_{0}^{i\infty} \frac{ g_{\frac14,0}\left(u\right)}{\sqrt{-i( u+\tau )}}du  -\frac{i}{2} \sqrt{-i\tau_2}, \\
 \mathcal I_2(\tau) &:=& \displaystyle \frac12 \sqrt{-i\tau_1}\ h\left(\tfrac{1}{4}  ;\tau_1\right) &=& \displaystyle\frac{i}{2} \sqrt{\tau+1 } \int_{1}^0 \frac{  g_{\frac14,\frac12}(u)  }{\sqrt{-i(u+\tau)}}du ,  \\
 \mathcal J_2(\tau) &:=&\displaystyle \frac{-\zeta_8}{2}q^{-\frac{1}{32}}\sqrt{\tau+1 }  \ h\left(\tfrac{ \tau}{4};\tau\right)  &=& \displaystyle\frac{i}{2}\sqrt{\tau+1} \int_{0}^{i\infty} \frac{ g_{\frac14,\frac12}(u)  }{\sqrt{-i(u+\tau)}}du,\\
 \mathcal I_3(\tau) &:=&\displaystyle -\frac{\zeta_6}{2i}   e\left(\tfrac{1}{8\tau}\right) \sqrt{-i\tau_2} \ h\left( \tfrac{\tau_2}{2}+\tfrac16;\tau_2\right) &=& \displaystyle\frac{i}{2}  \sqrt{  2\tau+1 }\displaystyle  \int_{\frac12}^{0}  \frac{ g_{\frac13,0}\left(u\right)}{\sqrt{-i(u+\tau)}}du  + \frac{i}{2} \sqrt{ -i\tau_2}, \\
 \mathcal J_3(\tau) &:=&\displaystyle \frac{1}{2i} q^{-\frac{1}{72}}\sqrt{2\tau+1} \  h\left(\tfrac{\tau}{6}-\tfrac12;\tau\right) &=&  \displaystyle \frac{i}{2 }  \sqrt{2\tau+1}   \displaystyle \int_{0}^{i\infty} \frac{ g_{\frac13,0}\left(u\right)}{\sqrt{-i( u+\tau )}}du  -\frac{i}{2} \sqrt{-i\tau_2}, \\
 \mathcal I_4(\tau) &:= &\displaystyle \frac12 \sqrt{-i\tau_1}\left(h\left(\tfrac{5}{12}  ;\tau_1\right)+h\left(\tfrac{1}{12};\tau_1\right)\right)
 &=& \displaystyle \frac{i\zeta_8}{2}\sqrt{\tau+1} \int_{1}^0 \frac{ \zeta_{24}^{-1}g_{\frac{1}{12},\frac12}(u) + \zeta_{24}^{-5} g_{\frac{5}{12},\frac12}(u)}{\sqrt{-i(u+\tau)}}du, \\
 \mathcal J_4(\tau) &:=&\displaystyle \frac{-\zeta_8}{2}\sqrt{\tau+1}\left( q^{\frac{-25}{288}} h\left(\tfrac{5\tau}{12};\tau\right) + q^{\frac{-1}{288}}h\left(\tfrac{\tau}{12};\tau\right)\right) &=& \displaystyle \frac{i\zeta_8}{2}\sqrt{\tau+1} \int_{0}^{i\infty} \frac{ \zeta_{24}^{-1}g_{\frac{1}{12},\frac12}(u) + \zeta_{24}^{-5} g_{\frac{5}{12},\frac12}(u)}{\sqrt{-i(u+\tau)}}du,  \\
  \mathcal I_5(\tau) &:=& \displaystyle  -\frac{\zeta_6^{-1}}{2}e\left(\tfrac{1}{8\tau}\right) \sqrt{-i\tau_2} \ h\left(\tfrac{\tau_2}{2}+\tfrac13  ;\tau_2\right)&=& \displaystyle   \frac{i}{2}\sqrt{ 2\tau+1 } \int_{\frac12}^0 \frac{ g_{\frac16,0}(u)}{\sqrt{-i(u+\tau)}}du + \frac{i}{2} \sqrt{-i\tau_2},
  \\
 \mathcal J_5(\tau) &:= & \displaystyle \frac{1}{2i}q^{\frac{-1}{18}}\sqrt{2\tau+1 } \ h\left(\tfrac{\tau}{3}-\tfrac12;\tau\right) &=& \displaystyle    \frac{i}{2}\sqrt{2\tau+1 } \int_{0}^{i\infty} \frac{ g_{\frac16,0}(u)}{\sqrt{-i(u+\tau)}}du - \frac{i}{2} \sqrt{-i\tau_2}, \\
  \mathcal I_6(\tau) &:=&\displaystyle \frac12 \sqrt{-i\tau_1}\ h\left(\tfrac{1}{6}  ;\tau_1\right) &=& \displaystyle\frac{i\zeta_{24}^{-1}}{2} \sqrt{\tau+1 } \int_{1}^0 \frac{ g_{\frac13,\frac12}(u)  }{\sqrt{-i(u+\tau)}}du ,  \\
 \mathcal J_6(\tau) &:=&\displaystyle \frac{1}{2i}\zeta_{8}^{-1}q^{-\frac{1}{72}}\sqrt{\tau+1 }  \ h\left(\tfrac{ \tau}{6};\tau\right)  &=& \displaystyle\frac{i\zeta_{24}^{-1}}{2}\sqrt{\tau+1} \int_{0}^{i\infty} \frac{ g_{\frac13,\frac12}(u)  }{\sqrt{-i(u+\tau)}}du,\\
\end{array} \end{align*} \captionof{table}{Mordell and Eichler integrals $\mathcal I_m$ and $\mathcal J_m$} \label{table_IJ} }

For $\tau \in \mathbb H$, the transformation law in part (i) of Theorem \ref{thm_quantum} when $m\in\{1,3,5\}$ and $n=1$, and the transformation law in part (ii) of Theorem \ref{thm_quantum} when $m \in \{2,4,6\}$ and $n=1$ (both of which pertain to $V_{m1}(M_{\ell_m}\tau),$ $m\in T$) now follow from (\ref{eqn_FEmij}), Table \ref{table_IJ}, and Lemma \ref{E-g}.  The transformation law in part (i) of Theorem \ref{thm_quantum} for $\tau \in \mathbb H$ when $m\in\{2,4,6\}$ and $n=1$ follows after a short calculation by iterating the transformation law given in part (ii),  applying Lemma \ref{lem_gabprop}, and simplifying.

\hspace{2ex} The transformation law (under $\tau \to \tau + b_m$) in part (iii) of Theorem \ref{thm_quantum} follows for $\tau \in \mathbb H$ by a direct calculation using Lemma \ref{lem_mu}.

\hspace{2ex} Having established parts (i), (ii), and (iii) of Theorem \ref{thm_quantum} for $\tau \in \mathbb H$ for $n=1$, we have continuation to $\tau=x \in S_{m1}\setminus\{-\frac{1}{2}\}$ in part (i), to $\tau=x \in S_{m1}\setminus\{-1\}$ in part (ii), and to $x\in S_{m1}$ in part (iii), by  Theorem \ref{thm_QuantumSet} and  the argument given in Section \ref{sec_quantumsets}.  As argued in \cite{BR, BOPR, FOR, ZagierV, Zagier}, for example, the integrals appearing in parts (i) and (ii) of Theorem \ref{thm_quantum} are real analytic functions, except at $-1/2$ and $-1$ (respectively).
\end{proof}
 
\begin{proof}[Detailed Proof of Theorem \ref{thm_quantum} for  $(m,n)=(1,1)$]
 As summarized in the Appendix  or (\ref{eq_Vnotation}), we may write
\begin{eqnarray}V_{11}(\tau)= -q^{-\frac{1}{32}}\mu(\tau/4 + 1/2, \tau/2 ; \tau).\label{V1tau}\end{eqnarray}
Thus, we have
   \begin{align}\notag V_{11}\left(M_2\tau\right)&=  -e\left(-\frac{1}{32}M_2\tau\right) \mu\left(\frac{S\tau_2}{4} + \frac12,\frac{S\tau_2}{2};S\tau_2\right)  \\
   \notag &= -e\left(-\frac{1}{32}M_2\tau\right)e\left(-\frac{1}{2\tau_2}\left(\frac14+\frac{\tau_2}{2} \right)^2\right) \\ \notag & \hspace{.3in} \times \sqrt{-i\tau_2} \left(-\mu\left(-\frac14 + \frac{\tau_2}{2},-\frac12;\tau_2\right) + \frac{1}{2i}h\left(\frac14 + \frac{\tau_2}{2};\tau_2\right)\right) \\
   \label{eqn_Ftrans} &= \zeta_8 e\left(\frac{1}{8\tau}\right)\sqrt{-i\tau_2}\mu\left(-\frac14 + \frac{\tau_2}{2},-\frac12;\tau_2\right) +  \mathcal I_1(\tau),
\end{align}
   where $$\mathcal I_1(\tau) := -\frac{1}{2i}\zeta_8  e\left(\frac{1}{8\tau}\right) \sqrt{-i\tau_2}h\left( \frac{\tau_2}{2}+\frac14;\tau_2\right).$$
   Here, we have used  Lemma \ref{lem_mu}(6).  Next, recalling that $\tau_2 = S^{-1}\tau - 2$, we (repeatedly) apply  Lemma \ref{lem_mu}(5) followed by a second application of  Lemma \ref{lem_mu}(6), as well as  Lemma \ref{lem_mu}(1, 3), and find after some simplification that (\ref{eqn_Ftrans}) equals
   \begin{align}\label{eqn_transFE1H} (2\tau+1)^{\frac12} V_{11}(\tau)  + \mathcal I_1(\tau) + \mathcal J_1(\tau),
   \end{align} where $$ \mathcal J_1(\tau) := \frac{1}{2i} e\left(-\frac{\tau}{32}\right)(2\tau+1)^{\frac12} h\left(\frac{\tau}{4}-\frac12;\tau\right).$$

We re-write $\mathcal I_1(\tau)$ using part (ii) of Lemma \ref{lem_Zlem}, and obtain after some simplification that $\mathcal I_1(\tau)$ equals
   \begin{align}\notag   -\frac{1}{2} &\sqrt{-i\tau_2} \int_{\frac12}^{0} \frac{g_{1,\frac14}\left(2-\frac{1}{u}\right)}{\sqrt{i(u^{-1}+\tau^{-1})}}\frac{du}{u^2}  + \frac{i}{2} \sqrt{-i\tau_2}  \\
  \notag &=   -\frac{1}{2} \sqrt{-i\tau_2} \int_{\frac12}^{0} \frac{g_{1,\frac14}\left(-\frac{1}{u}\right)}{\sqrt{i(u^{-1}+\tau^{-1})}}\frac{du}{u^2} + \frac{i}{2} \sqrt{-i\tau_2}\\
  \notag &=  -\frac{1}{2} \sqrt{-i\tau_2} (-i)^{\frac32} \int_{\frac12}^{0} \frac{g_{\frac34,1}\left(u\right)}{\sqrt{i(u^{-1}+\tau^{-1})}}\frac{du}{u^{\frac12}} + \frac{i}{2} \sqrt{-i\tau_2} \\
   \notag &=  -\frac{1}{2} \sqrt{i\tau_2\tau} (-i)^{\frac32} \int_{\frac12}^{0}\frac{g_{\frac34,1}\left(u\right)}{\sqrt{-i(u+\tau)}} du + \frac{i}{2} \sqrt{-i\tau_2} \\
 \label{eqn_transH1} &=   \frac{1}{2} \sqrt{  2\tau+1 }  \int_{\frac12}^{0} \frac{g_{\frac34,1}\left(u\right)}{\sqrt{-i(u+\tau)}}  du + \frac{i}{2} \sqrt{-i\tau_2}
   \end{align}
   where we have used Lemma \ref{lem_gabprop}.
  We also re-write $\mathcal J_1(\tau)$ using part (iii) of Lemma \ref{lem_Zlem}, and obtain after some simplification that $\mathcal J_1(\tau)$ equals
   \begin{align}\label{eqn_transH2} \frac{1}{2 }  \sqrt{2\tau+1}   \int_{0}^{i\infty} \frac{g_{\frac34,1}\left(u\right)}{\sqrt{-i( u+\tau )}}du  -\frac{i}{2} \sqrt{-i\tau_2}.
   \end{align}   Next, we re-write $g_{3/4,1}(z) = i g_{1/4,0}(\tau) =  \frac{i}{4}E_1(z/32)$, where we have used Lemma \ref{lem_gabprop} and Lemma \ref{E-g}.
   Combining this with (\ref{eqn_transFE1H}), (\ref{eqn_transH1}) and (\ref{eqn_transH2}) proves the first transformation law (under $\tau \to \tau/(2\tau+1)$) given in part (i) of  Theorem \ref{thm_quantum} for $\tau \in \mathbb H$.

 For the second transformation law (under $\tau \to \tau + 4$) in part (iii) of Theorem \ref{thm_quantum}, we again use Lemma \ref{lem_mu}.   From (\ref{V1tau}), we have
\begin{align*}
V_{11}(\tau+4)&=-\zeta_8^{-1}q^{\frac{1}{32}}\mu\left(\frac{\tau}{4}+1+\frac12,\frac{\tau}{2}+2,\tau+4\right)\\
&=-\-\zeta_8^{-1}q^{\frac{1}{32}}\mu\left(\frac{\tau}{4}+\frac12,\frac{\tau}{2},\tau\right)\\
&=\zeta_8^{-1}V_{11}(\tau),
\end{align*}
as desired. We have continuation to $\tau=x \in S_{11}\setminus\{-\frac12\}$ by  Theorem \ref{thm_QuantumSet} and   the argument in Section \ref{sec_quantumsets} .
         \end{proof}

\subsection{Proof of Theorem \ref{thm_quantum} for $n\neq 1$}\label{sec_qmfvmnall}
To prove the theorem in the remaining cases, we establish Lemma \ref{lem_modfmn} below, which shows that the auxiliary functions $\mathcal F_{mn},$ defined in \eqref{def_fmndef}, are weakly holomorphic modular forms, and provides explicit transformation properties.  
\begin{lemma}\label{lem_modfmn} The functions $ \mathcal F_{mn}$ are weakly holomorphic modular forms of weight $1/2$.  In particular, for $\tau \in \mathbb H$, the following are true.
 \ \\ \textnormal{(i)} \ For $m\in T' \setminus\{4\}$, for each admissible $n$, we have that $$
 \mathcal F_{mn}(\tau) +i^{\ell_m} (2\tau+1)^{-\frac12} \mathcal F_{mn}\left(\frac{\tau}{2\tau+1}\right) = 0.$$
 \noindent \textnormal{(ii)} \  For  $m \in T' \setminus \{4\}$, for each admissible $n$,  we have that $$\mathcal F_{mn}(\tau) - \zeta_{a_m}^{\kappa_{mn}}\mathcal F_{mn}(\tau+ \kappa_{mn} b_m)   = 0.$$
  \end{lemma}

We postpone the proof of Lemma \ref{lem_modfmn} until the end of this section, and first prove Theorem \ref{thm_quantum} for the remaining functions $V_{mn}$ (i.e. $n\neq 1$).
\begin{proof}[Proof of Theorem \ref{thm_quantum} for $n\neq 1$] We begin by re-writing the functions $V_{mn}$ using (\ref{eq_KangExtm}) and (\ref{eq_KangExt4}), which we previously established.    
Note that $\mathcal F_{m1}(\tau) $ is identically equal to zero for each $m\in T'\setminus\{4\}$.   We next use the fact that for fixed $m$ and each admissible $n$, we have that
$S_{mn} \subseteq S_{m1}$ and $G_{mn} \subseteq G_{m1}$.  Previously, in Section \ref{sec_quantumsets}, we showed that if $x\in S_{m1}$, then $Mx \in S_{m1}$ for any $M\in G_{m1}$.  A nearly identical argument shows that for fixed $m$ and each admissible $n$, that if $x\in S_{mn}$, then $Mx \in S_{mn}$ for any $M\in G_{mn}$.     Thus, for fixed $m$ and each admissible $n$ ($n\neq 1$),  the quantum modular transformation properties given in parts (i) and (iii) of Theorem \ref{thm_quantum} for the functions $V_{mn}$ with $n\neq 1$ now follow from the transformation properties established in Section \ref{sec_qs1}  for the functions $V_{m1}$ in Theorem \ref{thm_quantum} restricted to the subsets $S_{mn}\subseteq S_{m1}$ and the subgroups $G_{mn}\subseteq G_{m1}$, combined with Lemma \ref{lem_modfmn} and Lemma \ref{lem_zerofmn}.
This concludes the proof of Theorem \ref{thm_quantum} in the remaining cases $(n\neq 1)$.  \end{proof}

 \begin{proof}[Proof of Lemma \ref{lem_modfmn}]
\ \\ \emph{Proof of part (i).}  The proof of part (i) of Lemma \ref{lem_modfmn}  makes use of Lemma \ref{lem_etatransf} and Lemma \ref{lem_thetatransf}.   We divide our proof into six cases, corresponding to six possible values of $m$.  For  $m=1$, we give an explicit proof for each admissible $n$.  For the remaining cases ($m\in\{2,3,4',4'',5,6\}$), we provide a sketch of proof for brevity's sake, as the proofs in these cases are nearly identical to the case $m=1$.  To begin, we list some transformation properties of certain specialized Jacobi $\vartheta$-functions under $ M_2:= \sm{1}{0}{2}{1}$ which we will make use of:
 \begin{align}
 \label{eqn_thetaprops1}  \vartheta\left(\alpha M_2\tau + \frac12;M_2\tau\right) &= -iq^{-\frac12-\alpha} (2\tau+1)^{\frac12} e\left(\frac{\left((\alpha+1)\tau+\frac12\right)^2}{2\tau+1}\right) \vartheta\left(\alpha\tau + \frac12;\tau\right), \\
\label{eqn_thetaprops2}
\vartheta\left(\alpha M_2\tau ;M_2\tau\right) &=-i(2\tau+1)^{\frac12}e\left(\frac{\left(\alpha\tau\right)^2}{2\tau+1}\right) \vartheta\left(\alpha\tau;\tau\right),
  \end{align}
 where $\alpha \in \mathbb C$.  To establish (\ref{eqn_thetaprops1}) and (\ref{eqn_thetaprops2}), we have used Lemma \ref{lem_thetatransf}, and the fact that $\psi^3(M_2)=-i$.  \ \medskip \\ {\underline{\bf{Case} $\boldsymbol{m=1}$}.} \ \medskip  \ \\   
 We have by definition that $$\mathcal F_{1n}(\tau):=- i q^{-\frac{1}{32}}\frac{\eta^3(\tau) \vartheta\left(\frac{\tau}{2} + u_\tau^{(1n)};\tau\right)\vartheta\left(u_\tau^{(1n)} - \frac{\tau}{4}-\frac12;\tau\right)}{\vartheta\left(\frac{\tau}{4}+\frac12;\tau\right) \vartheta\left(\frac{\tau}{2};\tau\right) \vartheta(u_\tau^{(1n)};\tau) \vartheta(v_\tau^{(1n)};\tau)}.$$   Using transformation properties from (\ref{etatrgen}), (\ref{tt1}), (\ref{eqn_thetaprops1}), and (\ref{eqn_thetaprops2}), we find after some straightforward calculations that
 \begin{align}\label{eqn_f1nmodt}
 \mathcal F_{1n}(M_2\tau) = i (2\tau+1)^{\frac 12} \rho_{1n}(M_2) \mathcal F_{1n}(\tau),
 \end{align}
 where 
 \begin{align*} \rho_{12}(M_2) &:= q^{\frac{1}{32} + \frac54}  e\left(\frac{1}{2\tau+1}\left( \frac{-\tau}{32} \!+\! \left(\frac{3\tau}{4}\right)^2 \!+\! \left(\tau+\frac12\right)^2 \!-\! \left(\frac{5\tau}{4} + \frac12\right)^2 \!-\! \left(\frac{3\tau}{2}+\frac12\right)^2 \!-\! \left(\frac{\tau}{2}\right)^2 - \left(\frac{\tau}{4}\right)^2 \right)\right),
\\
 \rho_{13}(M_2) &:= q^{\frac{1}{32} + \frac14}  e\left(\frac{1}{2\tau+1}\left( \frac{-\tau}{32}  + \left(\frac{19\tau}{12}+\frac12\right)^2 + \left(\frac{\tau}{6} \right)^2 - \left(\frac{5\tau}{4}+\frac12\right)^2 - \left(\frac{\tau}{2}\right)^2- \left(\frac{\tau}{3}\right)^2 -\left(\frac{13\tau}{12}+\frac12\right)^2\right)\right), \\
  \rho_{14}(M_2) &:= q^{\frac{1}{32} + \frac{11}{12}}  e\left(\frac{1}{2\tau+1}\left( \frac{-\tau}{32}  + \left(\frac{7\tau}{12}\right)^2 + \left(\frac{7\tau}{6}+\frac12\right)^2 - \left(\frac{5\tau}{4} + \frac12\right)^2 - \left(\frac{\tau}{2}\right)^2 -  \left(\frac{\tau}{12}\right)^2 - \left(\frac{4\tau}{3}+\frac12\right)^2 \right)\right), \\
 \rho_{15}(M_2) &:= q^{\frac{1}{32} + \frac14}  e\left(\frac{1}{2\tau+1}\left( \frac{-\tau}{32}  + \left(\frac{3\tau}{2}+\frac12\right)^2 - \left(\frac{5\tau}{4} + \frac12\right)^2 - \left(\frac{\tau}{2}\right)^2 - \left(\tau+\frac12\right)^2 \right)\right), \\ 
  \rho_{17}(M_2) &:= q^{\frac{1}{32} + \frac{5}{12}}  e\left(\frac{1}{2\tau+1}\left( \frac{-\tau}{32}  + \left(\frac{17\tau}{12}+\frac12\right)^2 + \left(\frac{\tau}{3}\right)^2 - \left(\frac{5\tau}{4} + \frac12\right)^2 - \left(\frac{\tau}{2}\right)^2 -  \left(\frac{13\tau}{12}+\frac12\right)^2 - \left(\frac{ \tau}{6} \right)^2 \right)\right), \\
 \rho_{18}(M_2) &:= q^{\frac{1}{32} + \frac{7}{12}}  e\left(\frac{1}{2\tau+1}\left( \frac{-\tau}{32}  + \left(\frac{5\tau}{12}\right)^2+\left(\frac{4\tau}{3}+\frac12\right)^2 - \left(\frac{\tau}{2}\right)^2 - \left(\frac{\tau}{12}\right)^2 - \left(\frac{5\tau}{4}+\frac12\right)^2-\left(\frac{7\tau}{6}+\frac12\right)^2 \right)\right).
  \end{align*}  After simplifying, one finds that $\rho_{1n}(M_2)=-i$ for each admissible $n$.  Using this fact, Lemma \ref{lem_modfmn} follows from (\ref{eqn_f1nmodt}) for each $\mathcal F_{1n}$.     \ \medskip \\ {\underline{\bf{Case} $\boldsymbol{m\in\{2,3,4',4'',5,6\}}$}.} \ \medskip  \ \\   We proceed as above in the case $m=1$.  Using transformation properties from (\ref{eqn_thetaprops1}),  (\ref{eqn_thetaprops2}), and (\ref{tt1}), we find after some straightforward calculations that
 \begin{align}\label{eqn_f2nmodt}
 \mathcal F_{mn}(M_2\tau) = i (2\tau+1)^{\frac 12} \rho_{mn}(M_2) \mathcal F_{mn}(\tau),
 \end{align}
 where for any admissible $n$,
$$  \rho_{mn}(M_2) = \begin{cases} 1, & m \in\{2,4',4'',6\}, \\ -i, & m \in \{3,5\}. \end{cases}$$
 For example, for $(m,n) \in \{(2,2), (2,6), (3,4), (4',2), (4'',7), (5,6), (6,2)\},$ we have that
 \begin{align*} \rho_{22}(M_2) &:= q^{\frac{1}{32}}  e\left(\frac{1}{2\tau+1}\left( \frac{-\tau}{32}  \!+\! \left(\frac{7\tau}{4}+\frac12\right)^2  \!+\! \left(\tau+\frac12\right)^2  \!-\! \left(\frac{\tau}{4}  \right)^2 \!-\! \left(\frac{ \tau}{2} \right)^2 \!-\! \left(\frac{3\tau}{2}+\frac12\right)^2
 \!-\! \left(\frac{5\tau}{4}+\frac12\right)^2
  \right)\right) = 1,
\\
 \rho_{26}(M_2)  &:= q^{\frac{1}{32}- \frac12}  e\left(\frac{1}{2\tau+1}\left( \frac{-\tau}{32}  + \left(\frac{3\tau}{2}+\frac12\right)^2 - \left(\frac{\tau}{4} \right)^2 - \left(\frac{\tau}{2}\right)^2 - \left(\tau+\frac12\right)^2 \right)\right) = 1, \\\rho_{34}(M_2) &:= q^{\frac{1}{72} + 1}  e\left(\frac{1}{2\tau+1}\left( \frac{-\tau}{72}  \!+\! \left(\frac{2\tau}{3}\right)^2   \!+\! \left(\frac{7\tau}{6}+\frac12  \right)^2 \!-\! 2\left(\frac{4 \tau}{3}+\frac12 \right)^2 \!-\! \left(\frac{\tau}{2} \right)^2  - \left(\frac{\tau}{6}\right)^2\right)\right) = -i, 
  \\  \rho_{4'2}(M_2) &:= q^{\frac{25}{288} }  e\left(\!\frac{1}{2\tau+1}\!\left( \frac{-25\tau}{288}  \!+\!\left(\frac{19\tau}{12}\!+\!\frac12\right)^2   \!\!+\! \left( \tau\!+\!\frac12  \right)^2 \!\!-\!
 \left(\frac{\tau}{12}\right)^2 - \left(\frac{\tau}{2}\right)^2\!\!-\!
 \left(\frac{13 \tau}{12}\!+\!\frac12 \right)^2 \!\!-\! \left(\frac{3\tau}{2}\!+\!\frac12 \right)^2 \right)\!\right) \!=\! 1,
 \\
  \rho_{4''7}(M_2) &:= q^{\frac{1}{288} }  e\left(\frac{1}{2\tau+1}\left( \frac{-1\tau}{288}   \!+\! \left(\frac{7\tau}{12} \right)^2   \!+\! \left( \frac{\tau}{3}  \right)^2 \!-\!
 \left(\frac{5\tau}{12}\right)^2 - \left(\frac{\tau}{2}\right)^2 -
 \left(\frac{ \tau}{12}  \right)^2 \!-\! \left(\frac{\tau}{6} \right)^2 \right)\right) = 1,
 \\  \rho_{56}(M_2)  &:= q^{\frac{1}{18} + \frac{2}{3}}  e\left(\frac{1}{2\tau+1}\left( \frac{-\tau}{18}  \!+\! \left(\frac{5\tau}{12}\right)^2   \!-\! \left(\frac{7 \tau}{6}+\frac12 \right)^2 \!-\! \left(\frac{\tau}{2} \right)^2  - \left(\frac{\tau}{12}\right)^2 \right)\right) = -i, \\
  \rho_{62}(M_2) &:= q^{\frac{1}{72} }  e\left(\frac{1}{2\tau+1}\left( \frac{-\tau}{72}   \!+\! \left(\frac{11\tau}{6}\!+\!\frac12\right)^2   \!\!\!+\!  \left( \tau\!+\!\frac12  \right)^2 \!\!\!-\!
 \left(\frac{\tau}{3}\right)^2 \!\!\!-\! \left(\frac{\tau}{2}\right)^2 \!\!\!-\!
 \left(\frac{3 \tau}{2}\!+\!\frac12 \right)^2 \!\!\!-\! \left(\frac{4\tau}{3}\!+\!\frac12 \right)^2 \right)\right) \!=\! 1.
\end{align*} 
 \medskip \ \\ \emph{Proof of part (ii).} The proof in this case also follows by direct calculations using the definition of the functions $\mathcal F_{mn}$, as well the transformations
  \begin{align}\label{eqn_etvart} \eta(\tau + b)  = \zeta_{24}^b \eta(\tau), \ \ \ \
  \vartheta(z+a;\tau+b)  = (-1)^a \zeta_8^b \vartheta(z;\tau), \end{align} which hold for any $a,b \in \mathbb Z$, and follow from (\ref{etatrgen}), (\ref{tt1}), and (\ref{tt2}).  We provide details in the cases $m\in\{1,2,3,6\}$ and leave the remaining cases $m\in\{4',4'',5\}$ to the reader for brevity, as the proofs follow in a similar manner.
 \medskip \\ {\underline{\bf{Case $\boldsymbol{m\in\{1,2\}}$.}}}  In this case, $b_m=4, a_m=8$, and $t_m=-1/32$.  Using (\ref{eqn_etvart}) and a direct calculation, we find that the portion of $\mathcal F_{mn}$ independent of $\vartheta$-functions satisfies  \begin{align}\label{eqn_shift1} iw_m\eta^3(\tau+\kappa_{mn}b_m)e^{2\pi i t_m(\tau+\kappa_{mn}b_m)} = -\zeta_{a_m}^{-\kappa_{mn}} \cdot i w_m \eta^3(\tau)q^{t_m}.\end{align}    Thus, it suffices to show that under $\tau \mapsto \tau + \kappa_{mn}b_m$, the functions $\mathcal F_{mn}(\tau)/(i w_m \eta^3(\tau)q^{t_m})$, which are quotients of $\vartheta$-functions, map to $-\mathcal F_{mn}/(i w_m \eta^3(\tau)q^{t_m}).$
We compute using the definitions of $u_\tau^{(mn)}$ and $v_\tau^{(mn)}$ that
\begin{align*}u_{\tau+\kappa_{1n} b_1}^{(1n)}  &= \begin{cases}
u_{\tau}^{(1n)} \pm 1, & n\in\{2,3,4,7,8\}, \\
u_{\tau}^{(1n)}, & n= 5, \end{cases} \ \ \ \ v_{\tau+\kappa_{1n} b_1}^{(1n)} = \begin{cases}
v_{\tau}^{(1n)} +4, & n\in\{3,4\}, \\
v_{\tau}^{(1n)} + 2, & n\in\{2,7,8\},  \\
v_{\tau}^{(1n)} + 1, & n=5, \end{cases}
\\ u_{\tau+\kappa_{2n} b_2}^{(2n)}  &= \begin{cases}
u_{\tau}^{(2n)} \pm 1, & n\in\{2,3,4,7,8\}, \\
u_{\tau}^{(2n)}, & n= 6, \end{cases} \ \ \ \ v_{\tau+\kappa_{2n} b_2}^{(2n)} = \begin{cases}
v_{\tau}^{(2n)} +4, & n\in\{3,4\}, \\
v_{\tau}^{(2n)} + 2, & n\in\{2,7,8\},  \\
v_{\tau}^{(2n)} + 1, & n=6, \end{cases} \\ u_{\tau+\kappa_{1n} b_1}^{(11)}  &= \begin{cases}u_{\tau}^{(11)}+ 3, & n\in\{3,4,7,8\}, \\
u_{\tau}^{(11)} +1, & n\in\{2,5\}, \end{cases} \ \ \ \ \ \ \
u_{\tau+\kappa_{2n} b_2}^{(21)}  = \begin{cases}u_{\tau}^{(21)}+ 3, & n\in\{3,4,7,8\}, \\
u_{\tau}^{(21)} +1, & n\in\{2,6\}. \end{cases} \end{align*}
The claim now follows after combining the above with the transformation for the $\vartheta$-function given in (\ref{eqn_etvart}).
 \medskip \\ {\underline{\bf{Case $\boldsymbol{m\in\{3,6\}}$.}}}   In this case, $b_m=6, a_m=3,$ and $t_m=-1/72$.  In this case, analogous to (\ref{eqn_shift1}), we obtain
 $$  iw_m\eta^3(\tau+\kappa_{mn}b_m)e^{2\pi i t_m(\tau+\kappa_{mn}b_m)} = \zeta_{a_m}^{-\kappa_{mn}} \cdot i w_m \eta^3(\tau)q^{t_m}.$$    Thus, it suffices to show that under the functions $\mathcal F_{mn}(\tau)/(i w_m \eta^3(\tau)q^{t_m})$, which are quotients of $\vartheta$-functions, remain invariant under $\tau \mapsto \tau + \kappa_{mn}b_m$.  Using the definititions of $u_\tau^{(mn)}$ and $v_\tau^{(mn)},$ we find that
 \begin{align*}u_{\tau+\kappa_{3n} b_3}^{(3n)}  &= \begin{cases}
u_{\tau}^{(3n)} +2, & n=2, \\
u_{\tau}^{(3n)}+1, & n\in\{3,4,5,6\}, \\
u_\tau^{(3n)}, & n=7,
 \end{cases} \ \ \ \ v_{\tau+\kappa_{3n} b_3}^{(3n)} = \begin{cases}
v_{\tau}^{(3n)} +3, & n\in\{2,5,6\}, \\
v_{\tau}^{(3n)} + 2, & n\in\{3,4\},  \\
v_{\tau}^{(3n)} + 1, & n=7, \end{cases}
\\ u_{\tau+\kappa_{6n} b_6}^{(6n)}  &= \begin{cases}
u_{\tau}^{(6n)} +2, & n=2, \\
u_{\tau}^{(6n)}+1, & n\in\{3,4,5,6\}, \\
u_\tau^{(6n)}, & n=8,
 \end{cases} \ \ \ \ v_{\tau+\kappa_{6n} b_6}^{(6n)} = \begin{cases}
v_{\tau}^{(6n)} +3, & n\in\{2,5,6\}, \\
v_{\tau}^{(6n)} + 2, & n\in\{3,4\},  \\
v_{\tau}^{(6n)} + 1, & n=8, \end{cases} \\ u_{\tau+\kappa_{3n} b_3}^{(31)}  &= \begin{cases}u_{\tau}^{(31)}+ 2, & n\in\{2,3,4,7\}, \\
u_{\tau}^{(31)} +4, & n\in\{5,6\}, \end{cases} \ \ \ \
u_{\tau+\kappa_{6n} b_6}^{(61)}  = \begin{cases}u_{\tau}^{(61)}+ 2, & n\in\{2,3,4,8\}, \\
u_{\tau}^{(61)} +4, & n\in\{5,6\}. \end{cases} \end{align*}
The claim now follows after combining the above with the transformation for the $\vartheta$-function given in (\ref{eqn_etvart}).
  \end{proof}
\subsection{Proof of Proposition \ref{prop_eichlere}}\label{sec_qeichlere}
 We follow a method of proof and argument originally due to Zagier in \cite{Zagier}, which was later generalized  in \cite{BR}, and used also   in \cite{FKTY},  for example;  we refer the reader to these sources for more explicit details, and provide a detailed sketch of proof here.  The functions $E_m$ are modular forms of weight $3/2$, and satisfy, for all $\tau \in \mathbb H$ and  $\gamma=\sm{a}{b}{c}{d} \in \mathcal G_m \subseteq  \textnormal{SL}_2(\mathbb Z)$, the transformation
\begin{align}\label{eqn_Emmod} E_m(\gamma \tau)  = \nu_m(\gamma) (c\tau+d)^{\frac32} E_m(\tau).\end{align} Here, $\nu_m$ and $\mathcal G_m$ are suitable multipliers and subgroups (respectively), and can be explicitly determined using the definitions of the functions $E_m$.
We define the function $E_m^*(-\tau)$ ($\tau \in \mathbb H$) by
$$E_m^*(-\tau):=\int_{-\overline{\tau}}^{i\infty} \frac{E_m(u)}{\sqrt{u+\tau}}du.$$  Using (\ref{eqn_Emmod}), it is not difficult to show for all $\tau \in \mathbb H$ and $\gamma \in \mathcal G_m$ that
\begin{align}\label{eqn_Emstar} E_m^*( -\tau) -(-c\tau + d)^{-\frac12}\nu_m^{-1}(\gamma)E_m^*(\gamma(-\tau)) =  \int_{-\frac{d}{c}}^{i\infty} \frac{E_m(u)}{\sqrt{u+\tau}}du.\end{align}   Under a change of variable in the integrand, with an appropriate choice of matrix $\gamma$, up to multiplication by a constant (which can be explicitly determined), we find that the transformations given in (\ref{eqn_Emstar}) for the functions $E_m^*\left(-{2\tau}/{c_m^2}\right)$ are identical to the transformations given for the functions $V_{mn}(x)$ in Theorem \ref{thm_quantum} for $x\in S_{mn} \subseteq \mathbb Q$, as $\tau = x+iy \to x$ from the upper half-plane (as $y\to 0^+$), or equivalently, as
$z=-{2\tau}/{c_m^2} \to -{2x}/{c_m^2}$ from the lower half-plane.

On the other hand, we also have that the asymptotic expansions of $E_m^*(-\tau)$ and $\widetilde{E}_m(-\tau)$ agree at rational numbers $r/s$, that is, with $\tau=r/s + iy \in \mathbb H$, as $y\to 0^+$;  this fact is established  more generally in \cite[Proposition 2.1]{BR}.  Thus, the functions $\widetilde{E}_m$ inherit the transformation properties satisfied by the functions $E_m^*$ at appropriate rationals, and hence, transform (up to the aforementioned change of variable, up to a constant multiple) like the functions $V_{mn}$ in Theorem \ref{thm_quantum}, as claimed.

\section{Corollaries}\label{sec_cor}
In this section, we prove Corollary \ref{cor_hypergeometric}, in which we evaluate Eichler integrals of eta-theta functions $E_m$  appearing in Theorem \ref{thm_quantum} as finite $q$-hypergeometric sums, and establish related curious algebraic identities.  We define for $m\in \{1,2,3,5,6\}$ the numbers $d_m$ by
 $$d_m := \begin{cases}3, & m=1,2, \\ 1, & m=3,6,  
\\ 5, & m=5. \end{cases}$$
For $h/k \in \mathbb Q$ with $\gcd(h,k)=1$, we define for positive integers $m$ the numbers \begin{align}\label{def_HK} H_m = H_m (h,k) &:= \begin{cases} h, & mh+k > 0, \\ -h, & mh+k < 0  \end{cases}, \\ K_m = K_m (h,k) &:= |mh+k|. \nonumber \end{align}
\begin{proof}[Proof of Corollary \ref{cor_hypergeometric}]
We first establish (\ref{eqn_cor1}) and (\ref{eqn_cor2}).  To do so, we begin with parts (i) and (ii) of Theorem \ref{thm_quantum} in the case $n=1$.  We then use Lemma \ref{Vm1} to re-write the functions $V_{m1}$ in terms of the functions $f_m$ and $g_2$.  By Lemma \ref{lem_eta}, we have that the functions $f_m$ vanish at rationals in $S_{m1}$.  From Lemma \ref{lem_g_2}, we also have that the remaining functions in Lemma \ref{Vm1}, defined using the function $g_2$, are defined at rationals in $S_{m1}$. Moreover, the proof of Lemma \ref{lem_g_2} more specifically reveals that the functions defined using the infinite sums $g_2$ in Lemma \ref{Vm1} in fact truncate, and become finite sums.  Identities (\ref{eqn_cor1}) and (\ref{eqn_cor2}) of Corollary \ref{cor_hypergeometric} then follow by a direct calculation using the definition of the functions $F_{h,k}$ given in (\ref{def_Fhk}), and the numbers $H_m$ and $K_m$ in (\ref{def_HK}).  The claimed identities in (\ref{eqn_moreover1}) and (\ref{eqn_moreover2}) follow similarly.  We begin with part (iii) of Theorem \ref{thm_quantum} in the case $n=1$, then apply Lemma \ref{Vm1},  Lemma \ref{lem_eta}, and Lemma \ref{lem_g_2}.

\end{proof}
\section{Acknowledgements}
This research was supported by the American Institute of Mathematics through their SQuaREs (Structured Quartet Research Ensembles) program.  The authors are deeply grateful to AIM, and in particular to Brian Conrey and Estelle Basor, for the generous support and consistent encouragement throughout this project.  
The first author is additionally grateful for the support of NSF CAREER grant DMS-1449679, and for the hospitality provided by the Institute for Advanced Study, Princeton, under NSF grant DMS-128155.

\section*{Appendix}\label{sec_appendix}
\setcounter{table}{0}
\renewcommand{\thetable}{E\arabic{table}}

Here we tabulate all of our mock modular forms $V_{mn}$, for any admissible pair $(m,n)$, as originally defined as quotients of Lambert-type series and the eta-theta functions $e_n$, and also in terms of Zwegers' $\mu$-function.    These functions have \emph{normalized shadow} $E_m(\tau)$, meaning their shadows are equal to a constant multiple of the function $E_m(2\tau/c_m^2)$, where $c_m$ is defined in Section \ref{sec_quantum}.    We note that embedded in these tables are the definitions of the constants $w_m$, $t_m$, $u^{(mn)}_\tau$, and $v^{(mn)}_\tau$.

\begin{table}[h!]
\begin{center}
\caption{Mock theta functions with  normalized shadow $ E_1(\tau)$, where $u^{(1n)}_\tau-v^{(1n)}_\tau = -\frac14 \tau+\frac12$.}
\bgroup
\def\arraystretch{2}
\begin{tabular}{|c|c|c|}
\hline
$V_{1n}(\tau)$ & Series& $w_1 q^{t_1} \mu \left( u^{(1n)}_\tau,v^{(1n)}_\tau;\tau \right)$   \\
\hline
$V_{11}(\tau)$ & $\displaystyle{\frac{q^{-9/32}}{e_1(\tau/2)}\sum_{n\in\Z}\frac{(-1)^nq^{(n+1)^2/2}}{1+q^{n+1/4}}}$& $-q^{-1/32}\mu(\frac{\tau}{4}+\frac{1}{2}, \frac{\tau}{2}; \tau)$ \\
\hline
$V_{12}(\tau)$ & $\displaystyle{\frac{-q^{-9/32}}{e_2(\tau/2)}\sum_{n\in\Z}\frac{q^{(n+1)^2/2}}{1-q^{n+1/4}}}$& $-q^{-1/32}\mu(\frac{\tau}{4}, \frac{\tau}{2}-\frac{1}{2}; \tau)$ \\
\hline
$V_{13}(\tau)$ & $\displaystyle{\frac{q^{-9/32}}{e_3(\tau/72)}\sum_{n\in\Z}\frac{(-1)^nq^{(n+5/6)^2/2}}{1+q^{n+1/12}}}$& $-q^{-1/32}\mu(\frac{\tau}{12}+\frac{1}{2}, \frac{\tau}{3}; \tau)$ \\
\hline
$V_{14}(\tau)$ & $\displaystyle{\frac{-q^{-9/32}}{e_4(\tau/72)}\sum_{n\in\Z}\frac{q^{(n+5/6)^2/2}}{1-q^{n+1/12}}}$& $-q^{-1/32}\mu(\frac{\tau}{12}, \frac{\tau}{3}-\frac{1}{2}; \tau)$ \\
\hline
$V_{15}(\tau)$ & $\displaystyle{\frac{q^{-9/32}}{e_5(\tau/32)}\sum_{n\in\Z}\frac{(-1)^nq^{(n+3/4)^2/2}}{1+q^{n}}}$& $-q^{-1/32}\mu(\frac{1}{2}, \frac{\tau}{4}; \tau)$ \\
\hline
$V_{16}(\tau)$ &  --- & $-q^{-1/32}\mu(0, \frac{\tau}{4}-\frac{1}{2}; \tau)$ \\
\hline
$V_{17}(\tau)$ & $\displaystyle{\frac{q^{-9/32}}{e_7(\tau/18)}\sum_{n\in\Z}\frac{(-1)^nq^{(n+2/3)^2/2}}{1+q^{n-1/12}}}$& $-q^{-1/32}\mu(-\frac{\tau}{12}+\frac{1}{2}, \frac{\tau}{6}; \tau)$ \\
\hline
$V_{18}(\tau)$ & $\displaystyle{\frac{-q^{-9/32}}{e_8(\tau/18)}\sum_{n\in\Z}\frac{q^{(n+2/3)^2/2}}{1-q^{n-1/12}}}$& $-q^{-1/32}\mu(-\frac{\tau}{12}, \frac{\tau}{6}-\frac{1}{2}; \tau)$ \\
\hline
\end{tabular}
\egroup
\end{center}
\end{table}

\begin{table}[h!]
\begin{center}
\caption{Mock theta functions with normalized shadow $E_2(\tau)$, where $u^{(2n)}_\tau-v^{(2n)}_\tau = -\frac14\tau$.}\bgroup
\def\arraystretch{2}
\begin{tabular}{|c|c|c|}
\hline
$V_{2n}(\tau)$ & Series & $w_2 q^{t_2} \mu \left( u^{(2n)}_\tau,v^{(2n)}_\tau;\tau \right)$   \\
\hline
$V_{21}(\tau)$ &$ \displaystyle{\frac{-q^{-9/32}}{e_1(\tau/2)}\sum_{n\in\Z}\frac{(-1)^nq^{(n+1)^2/2}}{1-q^{n+1/4}}}$&$ iq^{-1/32}\mu(\frac{\tau}{4}, \frac{\tau}{2}; \tau) $\\
\hline
$V_{22}(\tau)$ &$ \displaystyle{\frac{q^{-9/32}}{e_2(\tau/2)}\sum_{n\in\Z}\frac{q^{(n+1)^2/2}}{1+q^{n+1/4}}}$&$ iq^{-1/32}\mu(\frac{\tau}{4}-\frac{1}{2}, \frac{\tau}{2}-\frac{1}{2}; \tau) $\\
\hline
$V_{23}(\tau)$ &$ \displaystyle{\frac{-q^{-9/32}}{e_3(\tau/72)}\sum_{n\in\Z}\frac{(-1)^nq^{(n+5/6)^2/2}}{1-q^{n+1/12}}}$&$ iq^{-1/32}\mu(\frac{\tau}{12}, \frac{\tau}{3}; \tau) $\\
\hline
$V_{24}(\tau)$ &$ \displaystyle{\frac{q^{-9/32}}{e_4(\tau/72)}\sum_{n\in\Z}\frac{q^{(n+5/6)^2/2}}{1+q^{n+1/12}}}$&$ iq^{-1/32}\mu(\frac{\tau}{12}-\frac{1}{2}, \frac{\tau}{3}-\frac{1}{2}; \tau) $\\
\hline
$V_{25}(\tau)$ & ---& $iq^{-1/32}\mu(0, \frac{\tau}{4}; \tau)$ \\
\hline
$V_{26}(\tau)$ &$ \displaystyle{\frac{q^{-9/32}}{e_6(\tau/32)}\sum_{n\in\Z}\frac{q^{(n+3/4)^2/2}}{1+q^{n}}}$&$ iq^{-1/32}\mu(-\frac{1}{2}, \frac{\tau}{4}-\frac{1}{2}; \tau) $\\
\hline
$V_{27}(\tau)$ &$ \displaystyle{\frac{-q^{-9/32}}{e_7(\tau/18)}\sum_{n\in\Z}\frac{(-1)^nq^{(n+2/3)^2/2}}{1-q^{n-1/12}}}$&$ iq^{-1/32}\mu(-\frac{\tau}{12}, \frac{\tau}{6}; \tau) $\\
\hline
$V_{28}(\tau)$ &$ \displaystyle{\frac{q^{-9/32}}{e_8(\tau/18)}\sum_{n\in\Z}\frac{q^{(n+2/3)^2/2}}{1+q^{n-1/12}}}$&$ iq^{-1/32}\mu(-\frac{\tau}{12}-\frac{1}{2}, \frac{\tau}{6}-\frac{1}{2}; \tau) $\\
\hline
\end{tabular}
\egroup
\end{center}
\end{table}

\begin{table}[h!]
\begin{center}
\caption{Mock theta functions with normalized shadow $E_3(\tau)$, where $u^{(3n)}_\tau-v^{(3n)}_\tau = -\frac16\tau+\frac12$.}
\bgroup
\def\arraystretch{2}
\begin{tabular}{|c|c|c|}
\hline
$V_{3n}(\tau)$ & Series & $w_3 q^{t_3} \mu \left( u^{(3n)}_\tau,v^{(3n)}_\tau;\tau \right)$   \\
\hline
$V_{31}(\tau)$&$ \displaystyle{\frac{q^{-2/9}}{e_1(\tau/2)}\sum_{n\in\Z}\frac{(-1)^nq^{(n+1)^2/2}}{1+q^{n+1/3}}}$&$ -q^{-1/72}\mu(\frac{\tau}{3}+\frac{1}{2}, \frac{\tau}{2}; \tau) $\\
\hline
$V_{32}(\tau)$&$ \displaystyle{\frac{-q^{-2/9}}{e_2(\tau/2)}\sum_{n\in\Z}\frac{q^{(n+1)^2/2}}{1-q^{n+1/3}}}$&$ -q^{-1/72}\mu(\frac{\tau}{3}, \frac{\tau}{2}-\frac{1}{2}; \tau) $\\
\hline
$V_{33}(\tau)$&$ \displaystyle{\frac{q^{-2/9}}{e_3(\tau/72)}\sum_{n\in\Z}\frac{(-1)^nq^{(n+5/6)^2/2}}{1+q^{n+1/6}}}$&$ -q^{-1/72}\mu(\frac{\tau}{6}+\frac{1}{2}, \frac{\tau}{3}; \tau) $\\
\hline
$V_{34}(\tau)$&$ \displaystyle{\frac{-q^{-2/9}}{e_4(\tau/72)}\sum_{n\in\Z}\frac{q^{(n+5/6)^2/2}}{1-q^{n+1/6}}}$&$ -q^{-1/72}\mu(\frac{\tau}{6}, \frac{\tau}{3}-\frac{1}{2}; \tau) $\\
\hline
$V_{35}(\tau)$&$ \displaystyle{\frac{q^{-2/9}}{e_5(\tau/32)}\sum_{n\in\Z}\frac{(-1)^nq^{(n+3/4)^2/2}}{1+q^{n+1/12}}}$&$ -q^{-1/72}\mu(\frac{\tau}{12}+\frac{1}{2}, \frac{\tau}{4}; \tau) $\\
\hline
$V_{36}(\tau)$&$ \displaystyle{\frac{-q^{-2/9}}{e_6(\tau/32)}\sum_{n\in\Z}\frac{q^{(n+3/4)^2/2}}{1-q^{n+1/12}}}$&$ -q^{-1/72}\mu(\frac{\tau}{12}, \frac{\tau}{4}-\frac{1}{2}; \tau) $\\
\hline
$V_{37}(\tau)$&$
\displaystyle{\frac{q^{-2/9}}{e_7(\tau/18)}\sum_{n\in\Z}\frac{(-1)^nq^{(n+2/3)^2/2}}{1+q^{n}}}$&$ -q^{-1/72}\mu(\frac{1}{2}, \frac{\tau}{6}; \tau)$\\
\hline
$V_{38}(\tau)$& ---&$ -q^{-1/72}\mu(0, \frac{\tau}{6}-\frac{1}{2}; \tau)$ \\
\hline
\end{tabular}
\egroup
\end{center}
\end{table}

\begin{table}[h!]
\begin{center}
\caption{Mock theta functions with normalized shadow $E_4(\tau)$, where $u^{(4'n)}_\tau-v^{(4n)}_\tau = -\frac52\tau$ and $u^{(4''n)}_\tau-v^{(4n)}_\tau = -\frac{1}{12}\tau$.}
\bgroup
\def\arraystretch{2}
\begin{tabular}{|c|c|c|}
\hline
$V_{4n}(\tau) =$  & Series& $w_4 q^{t_4'} \mu \left( u^{(4'n)}_\tau,v^{(4n)}_\tau;\tau \right)$    \\
$V_{4'n}(\tau)+V_{4''n}(\tau)$ & & + $w_4 q^{t_4''} \mu \left( u^{(4''n)}_\tau,v^{(4n)}_\tau;\tau \right)$ \\
\hline
$V_{41}(\tau)=$ & $\displaystyle{\frac{-q^{-121/288}}{e_1(\tau/2)}\sum_{n\in\Z}\frac{(-1)^nq^{(n+1)^2/2}}{1-q^{n+1/12}}}$&$ iq^{-25/288}\mu(\frac{\tau}{12}, \frac{\tau}{2}; \tau)$ \\
$V_{4'1}(\tau)+ V_{4''1}\tau)$&$
 +\displaystyle{\frac{-q^{-49/288}}{e_1(\tau/2)}\sum_{n\in\Z}\frac{(-1)^nq^{(n+1)^2/2}}{1-q^{n+5/12}}}$& $+iq^{-1/288}\mu(\frac{5\tau}{12}, \frac{\tau}{2}; \tau)$\\
 \hline
$V_{42}(\tau)=$ &$ \displaystyle{\frac{q^{-121/288}}{e_2(\tau/2)}\sum_{n\in\Z}\frac{q^{(n+1)^2/2}}{1+q^{n+1/12}}}$&$ iq^{-25/288}\mu(\frac{\tau}{12}-\frac{1}{2}, \frac{\tau}{2}-\frac{1}{2}; \tau)
$\\
$V_{4'2}(\tau) + V_{4''2}(\tau) $&
+$\displaystyle{\frac{q^{-49/288}}{e_2(\tau/2)}\sum_{n\in\Z}\frac{q^{(n+1)^2/2}}{1+q^{n+5/12}}}$&
$+iq^{-1/288}\mu(\frac{5\tau}{12}-\frac{1}{2}, \frac{\tau}{2}-\frac{1}{2}; \tau)$ \\
\hline
$V_{43}(\tau)=$ & $\displaystyle{\frac{-q^{-121/288}}{e_3(\tau/72)}\sum_{n\in\Z}\frac{(-1)^nq^{(n+5/6)^2/2}}{1-q^{n-1/12}}}$&$ iq^{-25/288}\mu(-\frac{\tau}{12}, \frac{\tau}{3}; \tau)$ \\
$V_{4'3}(\tau) + V_{4''3}(\tau) $&
 $+\displaystyle{\frac{-q^{-49/288}}{e_3(\tau/72)}\sum_{n\in\Z}\frac{(-1)^nq^{(n+5/6)^2/2}}{1-q^{n+1/4}}}$& $+iq^{-1/288}\mu(\frac{\tau}{4}, \frac{\tau}{3}; \tau)$\\
 \hline
$V_{44}(\tau)=$ &
 $\displaystyle{\frac{-q^{-121/288}}{e_4(\tau/72)}\sum_{n\in\Z}\frac{q^{(n+5/6)^2/2}}{1+q^{n-1/12}}}$&$ iq^{-25/288}\mu(-\frac{\tau}{12}-\frac{1}{2}, \frac{\tau}{3}-\frac{1}{2}; \tau)$ \\
$V_{4'4}(\tau) + V_{4''4}(\tau) $ &
 $+\displaystyle{\frac{-q^{-49/288}}{e_4(\tau/72)}\sum_{n\in\Z}\frac{q^{(n+5/6)^2/2}}{1+q^{n+1/4}}}$& $+iq^{-1/288}\mu(\frac{\tau}{4}-\frac{1}{2}, \frac{\tau}{3}-\frac{1}{2}; \tau)$ \\
 \hline
$V_{45}(\tau)=$ &$ \displaystyle{\frac{-q^{-121/288}}{e_5(\tau/32)}\sum_{n\in\Z}\frac{(-1)^nq^{(n+3/4)^2/2}}{1-q^{n-1/6}}}$&$ iq^{-25/288}\mu(-\frac{\tau}{6}, \frac{\tau}{4}; \tau)$ \\
$V_{4'5}(\tau) + V_{4''5}(\tau)$&$
+\displaystyle{\frac{-q^{-49/288}}{e_5(\tau/32)}\sum_{n\in\Z}\frac{(-1)^nq^{(n+3/4)^2/2}}{1-q^{n+1/6}}}$&
$+iq^{-1/288}\mu(\frac{\tau}{6}, \frac{\tau}{4}; \tau) $\\
\hline
$V_{46}(\tau)=$ &$ \displaystyle{\frac{q^{-121/288}}{e_6(\tau/32)}\sum_{n\in\Z}\frac{q^{(n+3/4)^2/2}}{1+q^{n-1/6}}}$&$ iq^{-25/288}\mu(-\frac{\tau}{6}-\frac{1}{2}, \frac{\tau}{4}-\frac{1}{2}; \tau)$ \\
$V_{4'6}(\tau) + V_{4''6} (\tau)$&$
+\displaystyle{\frac{q^{-49/288}}{e_6(\tau/32)}\sum_{n\in\Z}\frac{q^{(n+3/4)^2/2}}{1+q^{n+1/6}}}$&$
+iq^{-1/288}\mu(\frac{\tau}{6}-\frac{1}{2}, \frac{\tau}{4}-\frac{1}{2}; \tau) $\\
\hline
$V_{47}(\tau)=$ &$
\displaystyle{\frac{-q^{-121/288}}{e_7(\tau/18)}\sum_{n\in\Z}\frac{(-1)^nq^{(n+2/3)^2/2}}{1-q^{n-1/4}}}$&$ iq^{-25/288}\mu(-\frac{\tau}{4}, \frac{\tau}{6}; \tau)$\\
$V_{4'7}(\tau) + V_{4''7}(\tau) $&$
+\displaystyle{\frac{-q^{-49/288}}{e_7(\tau/18)}\sum_{n\in\Z}\frac{(-1)^nq^{(n+2/3)^2/2}}{1-q^{n+1/12}}}$&$
+iq^{-1/288}\mu(\frac{\tau}{12}, \frac{\tau}{6};\tau)$\\
\hline
$V_{48}(\tau)=$ &$\displaystyle{\frac{q^{-121/288}}{e_8(\tau/18)}\sum_{n\in\Z}\frac{q^{(n+2/3)^2/2}}{1+q^{n-1/4}}}$&$ iq^{-25/288}\mu(-\frac{\tau}{4}-\frac{1}{2}, \frac{\tau}{6}-\frac{1}{2}; \tau) $\\
$V_{4'8}(\tau) + V_{4''8}(\tau) $&$
+\displaystyle{\frac{q^{-49/288}}{e_8(\tau/18)}\sum_{n\in\Z}\frac{q^{(n+2/3)^2/2}}{1+q^{n+1/12}}}$&$
+iq^{-1/288}\mu(\frac{\tau}{12}-\frac{1}{2}, \frac{\tau}{6}-\frac{1}{2}; \tau)$ \\
\hline
\end{tabular}
\egroup
\end{center}
\end{table}

\begin{table}[h!]
\begin{center}
\caption{Mock theta functions with normalized shadow $E_5(\tau)$, where $u^{(5n)}_\tau-v^{(5n)}_\tau = -\frac13\tau+\frac12$.}
\bgroup
\def\arraystretch{2}
\begin{tabular}{|c|c|c|}
\hline
$V_{5n}(\tau)$ & Series & $w_5 q^{t_5} \mu \left( u^{(5n)}_\tau,v^{(5n)}_\tau;\tau \right)$   \\
\hline
$V_{51}(\tau)$&$ \displaystyle{\frac{q^{-25/72}}{e_1(\tau/2)}\sum_{n\in\Z}\frac{(-1)^nq^{(n+1)^2/2}}{1+q^{n+1/6}}}$&$ -q^{-1/18}\mu(\frac{\tau}{6}+\frac{1}{2}, \frac{\tau}{2}; \tau) $\\
\hline
$V_{52}(\tau)$&$ \displaystyle{\frac{-q^{-25/72}}{e_2(\tau/2)}\sum_{n\in\Z}\frac{q^{(n+1)^2/2}}{1-q^{n+1/6}}}$&$ -q^{-1/18}\mu(\frac{\tau}{6}, \frac{\tau}{2}-\frac{1}{2}; \tau) $\\
\hline
$V_{53}(\tau)$&$ \displaystyle{\frac{q^{-25/72}}{e_3(\tau/72)}\sum_{n\in\Z}\frac{(-1)^nq^{(n+5/6)^2/2}}{1+q^{n}}}$&$ -q^{-1/18}\mu(\frac{1}{2}, \frac{\tau}{3}; \tau) $\\
\hline
$V_{54}(\tau)$& ---&$ -q^{-1/18}\mu(0, \frac{\tau}{3}-\frac{1}{2}; \tau)$ \\
\hline
$V_{55}(\tau)$&$ \displaystyle{\frac{q^{-25/72}}{e_5(\tau/32)}\sum_{n\in\Z}\frac{(-1)^nq^{(n+3/4)^2/2}}{1+q^{n-1/12}}}$&$ -q^{-1/18}\mu(-\frac{\tau}{12}+\frac{1}{2}, \frac{\tau}{4}; \tau) $\\
\hline
$V_{56}(\tau)$&$ \displaystyle{\frac{-q^{-25/72}}{e_6(\tau/32)}\sum_{n\in\Z}\frac{q^{(n+3/4)^2/2}}{1-q^{n-1/12}}}$&$ -q^{-1/18}\mu(-\frac{\tau}{12}, \frac{\tau}{4}-\frac{1}{2}; \tau) $\\
\hline
$V_{57}(\tau)$&$
\displaystyle{\frac{q^{-25/72}}{e_7(\tau/18)}\sum_{n\in\Z}\frac{(-1)^nq^{(n+2/3)^2/2}}{1+q^{n-1/6}}}$&$ -q^{-1/18}\mu(-\frac{\tau}{6}+\frac{1}{2}, \frac{\tau}{6}; \tau)$\\
\hline
$V_{58}(\tau)$&$\displaystyle{\frac{-q^{-25/72}}{e_8(\tau/18)}\sum_{n\in\Z}\frac{q^{(n+2/3)^2/2}}{1-q^{n-1/6}}}$&$ -q^{-1/18}\mu(-\frac{\tau}{6}, \frac{\tau}{6}-\frac{1}{2}; \tau) $\\
\hline
\end{tabular}
\egroup
\end{center}
\end{table}

\begin{table}[h!]
\begin{center}
\caption{Mock theta functions with normalized shadow $E_6(\tau)$, where $u^{(6n)}_\tau-v^{(6n)}_\tau = -\frac16\tau$.}\bgroup
\def\arraystretch{2}
\begin{tabular}{|c|c|c|}
\hline
$V_{6n}(\tau)$ & Series & $w_6 q^{t_6} \mu \left( u^{(6n)}_\tau,v^{(6n)}_\tau;\tau \right)$ \\
\hline
$V_{61}(\tau)$&$ \displaystyle{\frac{-q^{-2/9}}{e_1(\tau/2)}\sum_{n\in\Z}\frac{(-1)^nq^{(n+1)^2/2}}{1-q^{n+1/3}}}$ &$ iq^{-1/72}\mu(\frac{\tau}{3}, \frac{\tau}{2}; \tau) $\\
\hline
$V_{62}(\tau)$&$ \displaystyle{\frac{q^{-2/9}}{e_2(\tau/2)}\sum_{n\in\Z}\frac{q^{(n+1)^2/2}}{1+q^{n+1/3}}}$&$ iq^{-1/72}\mu(\frac{\tau}{3}-\frac{1}{2}, \frac{\tau}{2}-\frac{1}{2}; \tau) $\\
\hline
$V_{63}(\tau)$&$ \displaystyle{\frac{-q^{-2/9}}{e_3(\tau/72)}\sum_{n\in\Z}\frac{(-1)^nq^{(n+5/6)^2/2}}{1-q^{n+1/6}}}$&$ iq^{-1/72}\mu(\frac{\tau}{6}, \frac{\tau}{3}; \tau) $\\
\hline
$V_{64}(\tau)$&$ \displaystyle{\frac{q^{-2/9}}{e_4(\tau/72)}\sum_{n\in\Z}\frac{q^{(n+5/6)^2/2}}{1+q^{n+1/6}}}$&$ iq^{-1/72}\mu(\frac{\tau}{6}-\frac{1}{2}, \frac{\tau}{3}-\frac{1}{2}; \tau) $\\
\hline
$V_{65}(\tau)$&$ \displaystyle{\frac{-q^{-2/9}}{e_5(\tau/32)}\sum_{n\in\Z}\frac{(-1)^nq^{(n+3/4)^2/2}}{1-q^{n+1/12}}}$&$ iq^{-1/72}\mu(\frac{\tau}{12}, \frac{\tau}{4}; \tau) $\\
\hline
$V_{66}(\tau)$&$ \displaystyle{\frac{q^{-2/9}}{e_6(\tau/32)}\sum_{n\in\Z}\frac{q^{(n+3/4)^2/2}}{1+q^{n+1/12}}}$&$ iq^{-1/72}\mu(\frac{\tau}{12}-\frac{1}{2}, \frac{\tau}{4}-\frac{1}{2}; \tau) $\\
\hline
$V_{67}(\tau)$&---&$ iq^{-1/72}\mu(0, \frac{\tau}{6}; \tau)$ \\
\hline
$V_{68}(\tau)$&$\displaystyle{\frac{q^{-2/9}}{e_8(\tau/18)}\sum_{n\in\Z}\frac{q^{(n+2/3)^2/2}}{1+q^{n}}}$&$ iq^{-1/72}\mu(-\frac{1}{2}, \frac{\tau}{6}-\frac{1}{2}; \tau) $\\
\hline
\end{tabular}
\egroup
\end{center}
\end{table}

\clearpage
\bibliographystyle{amsplain}
%\bibliography{SQuaRE}

\end{document}